\documentclass{article}
\usepackage{psfrag}
\usepackage[parfill]{parskip}
\usepackage{float, graphicx}
\usepackage[]{epsfig}
\usepackage{amsmath, amsthm, amssymb,}
\usepackage{epsfig}
\usepackage{verbatim}
\usepackage{multicol}
\usepackage{url}

\usepackage{latexsym}
\usepackage{mathrsfs}
\usepackage[colorlinks, bookmarks=true]{hyperref}
\usepackage{graphicx, color}
\definecolor{vdarkblue}{rgb}{0,0.2,0.7}

\newcommand{\yukun}{\color{vdarkblue}}
\usepackage{graphicx}
\usepackage{amsmath}
\usepackage{enumerate}
\usepackage[normalem]{ulem}
\usepackage{bm}
\usepackage{dsfont}
\usepackage{stmaryrd}
\usepackage{enumitem}
\usepackage[noadjust]{cite}

\usepackage{indentfirst}
\usepackage{color}
\usepackage{xcolor}
\usepackage{hyperref}
\usepackage{mathtools}
\usepackage[margin=1.1in]{geometry}
\usepackage{cleveref}

\makeatletter

\def\@settitle{\begin{center}%
    \bfseries
 \normalfont\LARGE\@title
  \end{center}%
}
\def\@setauthors{\begin{center}%
 \normalsize\@author
  \end{center}%
}

\makeatother

\numberwithin{equation}{section}

\renewcommand{\cal}{\mathcal}
\newcommand\cA{{\mathcal A}}
\newcommand\cB{{\mathcal B}}

\newcommand{\cD}{{\cal D}}

\newcommand{\cF}{{\cal F}}
\newcommand{\cG}{{\cal G}}

\newcommand\cH{{\mathcal H}}

\newcommand{\cN}{{\cal N}}

\newcommand{\cX}{{\mathcal X}}
\newcommand{\cY}{{\mathcal Y}}

\newcommand{\bfS}{{\bf S}}

\newcommand{\fa}{{\mathfrak a}}
\newcommand{\fj}{{\mathfrak j}}
\newcommand{\fb}{{\mathfrak b}}
\newcommand{\fc}{{\mathfrak c}}

\newcommand{\fo}{{\mathfrak o}}

\newcommand{\fp}{{\mathfrak p}}

\newcommand{\fg}{{\mathfrak g}}

\newcommand{\bmu}{{\bm{u}}}
\newcommand{\bmv}{{\bm{v}}}

\newcommand{\fC}{{\mathfrak C}}

\newcommand{\bfm}{{\bf m}}

\newcommand{\be}{\begin{equation}}
\newcommand{\ee}{\end{equation}}

\newcommand{\rd}{{\rm d}}

\newcommand{\ri}{\mathrm{i}}

\newcommand{\bC}{{\mathbb C}}
\newcommand{\bE}{\mathbb{E}}

\newcommand{\bH}{\mathbb{H}}

\newcommand{\bP}{\mathbb{P}}

\newcommand{\bR}{{\mathbb R}}
\newcommand{\bX}{{\mathbb X}}
\newcommand{\bS}{\mathbb S}
\newcommand{\bT}{\mathbb T}
\newcommand{\bV}{\mathbb V}
\newcommand{\bW}{\mathbb W}

\newcommand{\bZ}{\mathbb{Z}}

\newcommand{\bI}{\mathbb{I}}

\newcommand{\al}{\alpha}

\newcommand{\la}{\lambda}

\DeclareMathOperator{\Tr}{Tr}

\DeclareMathOperator{\supp}{supp}
\DeclareMathOperator{\dist}{dist}

\DeclareMathOperator{\Har}{{Har}}

\DeclareMathOperator{\OO}{O}
\DeclareMathOperator{\oo}{o}

\renewcommand{\Re}{\mathop{\mathrm{Re}}}
\renewcommand{\Im}{\mathop{\mathrm{Im}}}

\newcommand{\deq}{\mathrel{\mathop:}=} 
 
\renewcommand{\leq}{\leqslant}
\renewcommand{\geq}{\geqslant}

\newcommand{\del}{\partial}

\newcommand{\wt}{\widetilde}

\newcommand{\qq}[1]{[\![{#1}]\!]}

\newcommand{\beq}{\begin{equation}}
\newcommand{\eeq}{\end{equation}}

\theoremstyle{plain} 
\newtheorem{theorem}{Theorem}[section]
\newtheorem*{theorem*}{Theorem}
\newtheorem{lemma}[theorem]{Lemma}
\newtheorem*{lemma*}{Lemma}

\newtheorem*{corollary*}{Corollary}
\newtheorem{proposition}[theorem]{Proposition}
\newtheorem*{proposition*}{Proposition}

\newtheorem*{assumption*}{Assumption}

\newtheorem{definition}[theorem]{Definition}
\newtheorem*{definition*}{Definition}

\newtheorem*{example*}{Example}
\newtheorem{remark}[theorem]{Remark}

\newtheorem*{remark*}{Remark}
\newtheorem*{remarks*}{Remarks}

\newcommand{\col}{\vcentcolon}

\newcommand{\msc}{m_{\rm sc}}
\newcommand{\md}{m_d}

\newcommand{{\bfa}}{{\mathbf a}}
\newcommand{{\bfG}}{{\mathbf G}}

\newcommand{\sfS}{{\sf S}}

\newcommand{\sfF}{{\sf F}}

\def\author#1{\par
    {\centering{\authorfont#1}\par\vspace*{0.05in}}
}

\def\titlefont{\fontsize{13}{15}\bfseries\boldmath\selectfont\centering{}}
\def\authorfont{\fontsize{13}{15}}

\let\affiliationfont\rhfont

\def\address#1{\par
    {\centering{\affiliationfont#1\par}}\par\vspace*{11pt}
}

\def\body{
\setcounter{footnote}{0}
\def\thefootnote{\alph{footnote}}
\def\@makefnmark{{$^{\rm \@thefnmark}$}}
}

\def\title#1{
    \thispagestyle{plain}
    \vspace*{-14pt}
    \vskip 79pt
    {\centering{\titlefont #1\par}}%
    \vskip 1em
}

\usepackage{scalerel}
\newcommand{\cT}{{\mathcal T}}

\definecolor{forestgreen}{RGB}{34, 139, 34}

\newcommand{\GG}{{\mathcal G}}
\newcommand{\fR}{{\mathfrak R}}

\newcommand{\tcG}{{\widetilde \cG}}
\newcommand{\oOmega}{\overline{\Omega}}

\begin{document}

\title{Gaussian Waves and Edge Eigenvectors of Random Regular Graphs}

\vspace{1.2cm}

\noindent   \begin{minipage}[c]{0.36\textwidth}
 \author{Yukun He}
\address{City University of Hong Kong
   yukunhe@cityu.edu.hk}
 \end{minipage}
\begin{minipage}[c]{0.31\textwidth}
 \author{Jiaoyang Huang}
\address{University of Pennsylvania\\
   huangjy@wharton.upenn.edu}
 \end{minipage}
\begin{minipage}[c]{0.31\textwidth}
 \author{Horng-Tzer Yau}
\address{Harvard University \\
   htyau@math.harvard.edu}
 \end{minipage}

\begin{abstract}
Backhausz and Szegedy (2019) demonstrated that the almost eigenvectors of random regular graphs converge to Gaussian waves with variance 
$0\leq \sigma^2\leq 1$. In this paper, we present an alternative proof of this result for the edge eigenvectors of random regular graphs, establishing that the variance must be $\sigma^2=1$. Furthermore, we show that the eigenvalues and eigenvectors are asymptotically independent. Our approach introduces a simple framework linking the weak convergence of the imaginary part of the Green's function to the convergence of eigenvectors, which may be of independent interest.
\end{abstract}

\bigskip
{
\hypersetup{linkcolor=black}
\setcounter{tocdepth}{1}
\tableofcontents
}

\newpage
\section{Introduction}\label{s:intro}

Given a graph $\mathcal{G}$ on $N$ vertices with an adjacency matrix $A$, the eigenvalues and eigenvectors of $A$ encode key structural and dynamical properties of the graph. The study of these spectral phenomena has a rich history, spanning over half a century, with numerous applications across mathematics, computer science, and physics.  For a broad overview of spectral graph theory, we refer readers to the foundational texts \cite{chung1997spectral, brouwer2011spectra}. The connection between eigenvalues and the expansion properties of graphs is explored in the survey article \cite{hoory2006expander}. Moreover, the applications of eigenvalues and eigenvectors in various algorithms, such as combinatorial optimization, spectral partitioning, and clustering, are discussed in \cite{mohar1997some, mohar1993eigenvalues, pothen1990partitioning, shi2000normalized,  spielman2012spectral, spielman2007spectral, spielman2007spectral, spielman1996spectral}.

In this paper, we study random $d$-regular graphs, sampled uniformly from all simple $d$-regular graphs on $N$ vertices. Recent work by the second and third named authors with Theo McKenzie \cite{huang2024ramanujan} established edge universality for these graphs, proving that the second largest and smallest eigenvalues asymptotically follow the Tracy-Widom$_1$ distribution associated with the Gaussian Orthogonal Ensemble (GOE) \cite{tracy1996orthogonal}. As a consequence, for large $N$, approximately $69\%$ of $d$-regular graphs are Ramanujan \cite{lubotzky1988ramanujan, margulis1988explicit}. Beyond eigenvalues, numerous properties of eigenvectors for both random and deterministic $d$-regular graphs have been extensively studied. Notable results include the delocalization of eigenvectors \cite{bauerschmidt2019local, huang2024spectrum, brooks2013non, alon2021high,ganguly2021non}, quantum ergodicity \cite{anantharaman2015quantum, anantharaman2017quantumanderson, anantharaman2019quantum, anantharaman2019recent, anantharaman2017quantum}, and the structure of nodal domains \cite{ganguly2023many, huang2020size, dekel2011eigenvectors}.

Here, we extend the edge universality of random $d$-regular graphs to include the coefficients of the \emph{eigenvectors}. Backhausz and Szegedy \cite{backhausz2019almost} previously showed that the \emph{almost eigenvectors} of random regular graphs converge to \emph{Gaussian waves}, which are Gaussian processes on infinite $d$-regular trees that satisfy eigenvalue equations and are invariant under automorphisms of the tree (see \cite{elon2009gaussian}), with variance $0 \leq \sigma^2 \leq 1$.  We provide an alternative proof for the edge eigenvectors of random regular graphs, demonstrating that the variance is necessarily $\sigma^2 = 1$. Furthermore, we establish that the edge eigenvalues and eigenvectors are asymptotically independent and jointly converge to the \emph{Airy$_1$ point process} (see \cite{tracy1996orthogonal, ramirez2011beta}) and independent copies of Gaussian waves. 

Our approach uses the local resampling technique (see \Cref{s:local_resampling}), as introduced in \cite{bauerschmidt2019local,huang2024spectrum}, to analyze Green's functions, which encode information about eigenvectors. By randomizing the boundary edges of large neighborhoods around a vertex $o$, we express the values of edge eigenvectors near $o$ as a Poisson kernel-type integral over these randomized boundary values. Averaging over a large number of such random values yields the asymptotic normality of eigenvectors.

 For Wigner-type random matrices, a series of works has established the universality of eigenvalue statistics. In the bulk, universality was demonstrated in \cite{bourgade2016fixed, erdHos2010bulk, erdHos2012rigidity, erdHos2011universality, erdHos2012bulk, erdHos2015gap, tao2011random, aggarwal2019bulk, erdos2010universality, aggarwal2018goe}, while edge universality was proven in \cite{erdHos2012rigidity, soshnikov1999universality, tao2010random, knowles2013eigenvector, lee2014necessary}. These results show that, after proper normalization and shifts, spectral statistics agree with those of the GOE. Additionally, the asymptotic normality of eigenvectors, both in the bulk and at the edge, has been established in \cite{knowles2013eigenvector, tao2012random, bourgade2017eigenvectormoment}. For Erdős–Rényi graphs $G(N, p)$ with $Np \geq N^{\oo(1)}$ and random $d$-regular graphs with $d \geq N^{\oo(1)}$, significant progress has also been made. The bulk eigenvalue universality \cite{bauerschmidt2017bulk, huang2015bulk, erdHos2013spectral, erdHos2012spectral}, edge eigenvalue universality \cite{lee2018local, he2021fluctuations, huang2020transition, huang2022edge, lee2024higher, bauerschmidt2020edge, huang2023edge, he2024spectral}, and the asymptotic normality of bulk eigenvectors \cite{bourgade2017eigenvector} have all been rigorously established.

Most of these results rely on comparison with GOE matrices or Gaussian divisible models (random matrices with a small Gaussian component). In contrast, our proof is more direct: we reveal the joint convergence of eigenvalues and eigenvectors by analyzing Green's functions. In a companion paper \cite{he2025extremal}, the first and second named authors, in collaboration with Chen Wang, further develop this method to prove the asymptotic normality of sparse Erd{\H o}s–R{\'e}nyi graphs in arbitrary directions and normal fluctuation in quantum ergodicity at the edge for Wigner matrices.

The random wave conjecture, proposed by Berry \cite{berry1977regular, berry1983semiclassical}, suggests that eigenfunctions of classically chaotic systems behave like Gaussian random fields in the high-energy limit. This conjecture has inspired a substantial body of research; for a comprehensive overview, see \cite{zelditch2017eigenfunctions}. Random $d$-regular graphs share fundamental characteristics with classically chaotic systems, including strong expansion properties and delocalized eigenvectors \cite{smilansky2013discrete}. Our result, which establishes the convergence of edge eigenvectors to Gaussian waves, provides a graph-theoretic analogue of Berry's conjecture for random $d$-regular graphs.

\subsection{Main Results}

Before stating the main results, we  introduce some necessary notations. The \emph{Airy$_1$ point process} introduced in \cite{tracy1996orthogonal}, is a fundamental object in random matrix theory, particularly associated with the GOE matrix. It is a point process $\{A_1\geq A_2\geq A_3\cdots\}$ on $\bR$, describing the limiting distribution of  eigenvalues near the spectral edge of a GOE matrix as the matrix size tends to infinity. The fluctuation of the rightmost particle is governed by the Tracy-Widom$_1$ distribution. More generally, for any $\beta>0$, Airy$_\beta$ point process, introduced in \cite{ramirez2011beta}, describes the edge scaling limit of the Gaussian $\beta$-ensemble, and is characterized as the eigenvalues of the stochastic Airy operator. 

For any $d\geq 3$, let $\cX_d$ denote the infinite $d$-regular tree, and $o$  as a distinguished vertex, referred to as the root. The vertex set of $\cX_d$ is denoted by $\bX_d$. An \emph{eigenvector process} with eigenvalue $\lambda$ is a joint distribution $\{\psi_\lambda(v)\}_{v\in \bX_d}$ of real valued random variables, each with variance $1$ such that the following eigenvector equation holds almost surely
\begin{align*}
\sum_{v\in \bX_d: v\sim o}\psi_\lambda (v)=\lambda \psi_\lambda(o),
\end{align*} 
where the summation is over all neighbors $v$ of $o$. Furthermore, the process is invariant under any automorphism of the $d$-regular tree.

A \emph{Gaussian wave}
is an eigenvector process whose joint distribution is Gaussian. It is established in \cite{elon2009gaussian} that for every
$-d\leq \lambda\leq d$ there exists a unique Gaussian wave $\psi_\la$ with eigenvalue $\lambda$. Let $L$ be the adjacency operator of the infinite $d$-regular tree $\cX_d$, and define  the resolvent operator $(L-z)^{-1}$ for $z\in \bC^+$. For any two vertices $i,j\in\bX_d$, the covariance structure of the Gaussian wave $\psi_\la$ is given by 
\begin{align*}
{\rm Cov}[\psi_\la(i) \psi_\la(j)]
&=\lim_{\varepsilon\rightarrow 0}\frac{\Im[(L-(\lambda+\ri \varepsilon))^{-1}_{ij}]}{\Im[(L-(\lambda+\ri \varepsilon))^{-1}_{oo}]},
\end{align*}
which is explicitly given as 
\begin{align*}
&{\rm Cov}[\psi_\la(i) \psi_\la(j)]=\frac{1}{(d-1)^{r/2}}\left(\frac{d-1}{d}U_r\left(\frac{\lambda}{2\sqrt{d-1}}\right)-\frac{1}{d}U_{r-2}\left(\frac{\lambda}{2\sqrt{d-1}}\right)\right),
\end{align*}
where $r=\dist_{\cX_d}(i,j)$ and $U_r(\theta)=\sin((r+1)\cos^{-1}(\theta))/\sin(\cos^{-1}(\theta))$ are the Chebyshev polynomials of the second kind. 

For the special case of edge eigenvectors, corresponding to $\lambda=2\sqrt{d-1}$. We denote the Gaussian wave by $\Psi=\psi_{2\sqrt{d-1}}$. The covariance structure simplifies to  
\begin{align}\label{e:edge_Gaussian_wave}
{\rm Cov}[\Psi(i) \Psi(j)]=\frac{1}{(d-1)^{r/2}}\left(\frac{d-1}{d}U_r(1)-\frac{1}{d}U_{r-2}(1)\right)=\frac{1}{(d-1)^{r/2}}\left(1+\frac{(d-2)r}{d}\right),
\end{align}
where the final expression follows from evaluating $U_r(1)=r+1$ and $U_{r-2}(1)=r-1$. In the sequel, for any $k\in \mathbb N_+$, we abbreviate $\qq{k}\deq \{1,2,...,k\}$.

\begin{theorem}
\label{t:ev_universality}
Fix $d\geq 3$, and set $\cA=d(d-1)/(d-2)^2$. Fix $k\geq 1$ and radius $r\geq 0$. Let ${A}$ be the adjacency matrix of random $d$-regular graphs $\cG$ on $N$ vertices, and set $H\deq {A}/\sqrt{d-1}$.  Let $\lambda_1 = {d}/{\sqrt{d-1}} \geq \lambda_2 \geq \cdots \geq \lambda_N$ denote the eigenvalues, and let $\bmu_1=\bm1/\sqrt{N}, \bmu_2, \cdots, \bmu_N$ be the associated normalized eigenvectors of $H$. Fix any vertex $o\in \qq{N}$. The rescaled extreme eigenvalues and the corresponding eigenvectors  $\{\pm\bmu_s\}_{2\leq s\leq k+1}$ (each multiplied by a random sign) when restricted to the radius-$r$ ball $\cB_r(o;\cG)$, jointly converge to the Airy$_1$ point process and independent Gaussian waves:

\begin{align*}\begin{split}
& ((\cA N)^{2/3} ( \lambda_{2} - 2 ), (\cA N)^{2/3} ( \lambda_{3} - 2 ),\cdots, (\cA N)^{2/3} ( \lambda_{k+1} - 2 ))\Rightarrow (A_1, A_2, A_3,\cdots, A_k),\\
&(\pm \sqrt{N}\bmu_2|_{\cB_r(o;\cG)}, \pm \sqrt{N}\bmu_3|_{\cB_r(o;\cG)},\cdots, \pm \sqrt{N}\bmu_{k+1}|_{\cB_r(o;\cG)})\Rightarrow (Z_1|_{\cB_r(o;\cX_d)}, Z_2|_{\cB_r(o;\cX_d)}\cdots, Z_k|_{\cB_r(o;\cX_d)}),
\end{split}
\end{align*}
where $A_1, A_2, \cdots$ are the points of the Airy$_1$ point process, and $Z_1, Z_2,\cdots, Z_{k}$ are independent copies of the Gaussian wave $\Psi$ from \eqref{e:edge_Gaussian_wave}. 
An analogous result holds for the smallest eigenvalues and their associated eigenvectors.
\end{theorem}

\begin{remark}
The joint convergence of the extreme eigenvalues to the Airy$_1$ point process follows from the edge universality result in \cite{huang2024ramanujan}. The new contributions in \Cref{t:ev_universality} are twofold: firstly the rescaled edge eigenvectors converge to independent Gaussian waves; secondly the extreme eigenvalues and the corresponding eigenvectors are asymptotically independent. \end{remark}

\begin{remark}
Eigenvectors are not unique and are defined only up to a sign. That is why we multiply each eigenvector by a random sign in \Cref{t:ev_universality}. Another way to express the convergence to Gaussian waves is through the quadratic expressions $\bmu_s(i)\bmu_s(j)$ for $i,j\in \bS$ where $\bS$ is the vertex set of $\cB_r(o;\cG)$. Specifically,
\begin{align*}
((\cA N)^{2/3}(\lambda_{s+1}-2), \{N\bmu_{s+1} (i)\bmu_{s+1} (j)\}_{i,j\in \bS} )_{1\leq s\leq k}\Rightarrow (A_s, \{Z_{s}(i)Z_{s}(j)\}_{i,j\in \bS})_{1\leq s\leq k},
\end{align*}
where $A_1\geq A_2\geq A_3\geq \cdots$ is the Airy$_1$ point process, and $\{Z_{s}\}_{1\leq s\leq k}$ are independent copies of the Gaussian wave $\Psi$ (recall from \eqref{e:edge_Gaussian_wave}).

\end{remark}

\subsection{Topological statements}

Let $\lambda_1 = {d}/{\sqrt{d-1}} \geq \lambda_2 \geq \cdots \geq \lambda_N$ denote the eigenvalues, and $\bmu_1=\bm1/\sqrt{N}, \bmu_2, \cdots, \bmu_N$ the associated normalized eigenvectors of the normalized adjacency matrix $H$ of random $d$-regular graphs $\cG$ on $N$ vertices.
The rescaled eigenvalues, given by $((\cA N)^{2/3}(\lambda_{1}-2), (\cA N)^{2/3}(\lambda_{2}-2),(\cA N)^{2/3}(\lambda_{3}-2),\cdots)$ can be encoded by as a random measure (or random point process) on $\bR$:
\begin{align*}
\sum_{s=1}^N\delta_{(\cA N)^{2/3}(\lambda_{s}-2)}.
\end{align*}
We define the Green's function and the Stieltjes transform of the empirical eigenvalue distribution respectively as 
\begin{align*}
G(z)=(H-z)^{-1}, \quad m_N(z)=\frac{1}{N}\Tr[(H-z)^{-1}]=\frac{1}{N}\sum_{s=1}^N \frac{1}{\lambda_s-z}, \quad z\in \bC^+.
\end{align*} 
The imaginary parts of $G(z)$ and $m_N(z)$ are harmonic functions on the upper half-plane. By zooming in around the spectral edge at $2$, we have
\begin{align*}
\frac{N^{1/3}}{\cA^{2/3}}\Im\left[G_{ij}\left(2+\frac{w}{(\cA N)^{2/3}}\right)\right]
&=\Im \sum_{s=1}^N \frac{N\bmu_s(i)\bmu_s(j)}{ (\cA N)^{2/3}(\lambda_s-2)- w},\\
\frac{N^{1/3}}{\cA^{2/3}}\Im\left[m_N\left(2+\frac{w}{(\cA N)^{2/3}}\right)\right]
&=\Im \sum_{s=1}^N \frac{1}{ (\cA N)^{2/3}(\lambda_s-2)- w}.
\end{align*}
These expressions are Poisson kernel integrals with respect to the measures $\sum_{s= 1}^N N\bmu_s(i)\bmu_s(j)\delta_{(\cA N)^{2/3}(\lambda_{s}-2)}$ and $\sum_{s= 1}^N\delta_{(\cA N)^{2/3}(\lambda_{s}-2)}$. 

In the remainder of this section, we review the topologies of locally finite measures and harmonic functions on the upper half-plane. 

\begin{definition}\label{d:vague}
For any locally finite measures $\mu_1, \mu_2, \cdots$ on $\bR$, we say that they \emph{converge in the vague topology} to another locally finite $\mu$, if $\mu_n(f)\to \mu(f)$, for any $f:\bR\to\bR$ that is compactly supported and smooth.  
\end{definition}

The imaginary part of Green's functions and Stieltjes transforms are harmonic functions on the upper half-plane. We denote the set of harmonic functions on the upper half-plane as $\Har(\mathbb{C}^+)$, equipped with the topology of uniform convergence on compact subsets. This is a separable complete metric space, as every Cauchy sequence in this space has a limit given by a harmonic function.

Let $\{f_n\}_{n \geq 1}$ be a locally equibounded sequence of harmonic functions in $\Har(\mathbb{C}^+)$, then there exists a subsequence of $\{f_n\}_{n \geq 1}$ that converges locally uniformly. Thus, the compact sets of $\Har(\mathbb{C}^+)$ are given by:
\[
\{f \in \Har(\mathbb{C}^+) : \max_{x \in K} |f(x)| \leq M\},
\]
for any compact subset $K$ of $\mathbb{C}^+$ and any large constant $M > 0$.

Under mild conditions, the Poisson integral of a measure on $\bR$ gives a harmonic function in $\Har(\mathbb{C}^+)$. The following lemma states that the convergence of harmonic functions in the sense of uniform convergence on compact subsets implies the vague convergence of the corresponding measures. We postpone its proof to \Cref{app:Green}.

\begin{lemma}\label{l:boundary_measure}
Let $\mu$ and $\{\mu_n\}_{n\geq 1}$ be a sequence of $\sigma$ finite Borel measure on $\bR$, such that
\begin{align*}
\max_{n\geq 1} \int_\bR \frac{\rd|\mu_n|(x)}{1+x^2}<\infty, \quad \int_\bR \frac{\rd|\mu|(x)}{1+x^2}<\infty.
\end{align*}
The following Poisson kernel integrals are well defined,
\begin{align*}
u_n(x+\ri y)= \frac{1}{\pi}\int_\bR \frac{y \rd \mu_n(s)}{(x-s)^2+y^2},\quad u(x+\ri y)= \frac{1}{\pi}\int_\bR \frac{y \rd \mu(s)}{(x-s)^2+y^2}.
\end{align*}
If $u_n$ converges to $u$ in the sense of uniform convergence on compact subsets, then $\mu_n$ converges to $\mu$ in vague topology. 
\end{lemma}

Throughout the remainder of this paper, the weak convergence of point processes is understood to mean convergence with respect to the vague topology, as defined in \Cref{d:vague}. Similarly, the weak convergence of harmonic functions on the upper half-plane refers to convergence with respect to the topology of uniform convergence on compact subsets.

\subsection{Proof ideas}

In this section, we present our new framework, which establishes the convergence of edge eigenvectors through the weak convergence of the imaginary part of the Green's function. We illustrate this framework using Wigner matrices. In the main body of the paper, we apply the same framework to prove the convergence of edge eigenvectors of random regular graphs to the Gaussian wave.

Let $W$ be an $N\times N$ real symmetric Wigner matrix. Specifically, the entries $W_{ij}=W_{ji}$ are real-valued, have mean zero and unit variance $\bE[W_{ij}]=0, \bE[W_{ij}^2]=1$, and possess bounded higher moments.  We denote the eigenvalues and corresponding normalized eigenvectors of $H:=N^{-1/2}W$ as 
$\lambda_1\geq \lambda_2\geq\cdots\geq \lambda_N$ and $\bmu_1, \bmu_2, \cdots, \bmu_N$ respectively.  Fix any finite index set $\bS\subset \qq{N}$. It was shown in \cite{knowles2013eigenvector} that as $N$ goes to infinity
\begin{align}\label{e:eiev_convergence}
(N^{2/3}(\lambda_s-2), \{N\bmu_s (i)\bmu_s (j)\}_{i,j\in \bS} )_{1\leq s\leq k}\Rightarrow (A_s, \{Z_{s}(i)Z_{s}(j)\}_{i,j\in \bS})_{1\leq s\leq k},
\end{align}
where $A_1\geq A_2\geq A_3\geq \cdots$ is the Airy$_1$ point process, and $\{Z_{s}(i)\}_{1\leq s\leq k, i\in \bS}$ are independent standard Gaussian variables $\cN(0,1)$. 
In what follows, we present a new proof of this result by leveraging the weak convergence of the imaginary part of the Green's function.

 We denote the Green's function of $H$ as $G(z)=(H-z)^{-1}$.
 The claim \eqref{e:eiev_convergence} follows from the following statement about the joint convergence of extreme eigenvalues, and the imaginary part of Green's functions at the edge. 
By the edge universality of Wigner matrices \cite{erdHos2012rigidity, soshnikov1999universality,tao2010random,knowles2013eigenvector, lee2014necessary}, the normalized extreme eigenvalues converge in weak topology to the Airy$_1$ point process
\begin{align}\label{e:edge_eig}
(N^{2/3} ( \lambda_{1} - 2 ), N^{2/3} ( \lambda_{2} - 2 ),\cdots, N^{2/3} ( \lambda_{k} - 2 ))\Rightarrow (A_1, A_2, A_3,\cdots, A_k).
\end{align}

For any $i,j\in \bS$, near the spectral edge $2$,  the imaginary part of the rescaled Green's functions $N^{1/3}\Im G_{ij}(2+wN^{-2/3})$ is explicitly given by 
\begin{align*}
N^{1/3}\Im G_{ij}(2+wN^{-2/3})=\Im \sum_{s=1}^N \frac{N\bmu_s(i)\bmu_s(j)}{ N^{2/3}(\lambda_s-2)- w}=\sum_{s=1}^N \frac{\Im[w]N\bmu_s(i)\bmu_s(j)}{ (N^{2/3}(\lambda_s-2)- \Re[w])^2+\Im[w]^2}.
\end{align*}
It is the Possion kernel integral with the measure 
\begin{align*}
\sum_{s=1}^N N\bmu_s(i)\bmu_s(j) \delta_{N^{2/3}(\lambda_s-2)},
\end{align*}
and harmonic for $w\in \bC^+$. The claim \eqref{e:eiev_convergence} follows from \Cref{l:boundary_measure}, and the following weak convergence of the imaginary part of the rescaled Green's functions at the edge
\begin{align}\label{e:edgeG_converge}
\{N^{1/3}(\Im[G_{ij}(2+wN^{-2/3})]\}_{i,j\in \bS}\Rightarrow \{\Im[\bfG_{ij}(w)]\}_{i,j\in \bS}, 
\end{align}
where 
\begin{align*}
\Im[\bfG_{ij}(w)]:=\sum_{s\geq 1}\frac{\Im[w]Z_s(i)Z_s(j)}{(A_s -\Re[w])^2+\Im[w]^2}, \quad i,j\in \bS,
\end{align*}
and $\{Z_s(i)\}_{i\in\bS, s\geq 1}$ are independent Gaussian random variables. 

In the following, we give a sketch for the proof of \eqref{e:edgeG_converge}. We recall the semicircle distribution $\varrho_{\rm sc}(x)$ and its Stieltjes transform $\msc(z)$:
\begin{align}\begin{split}\label{e:msc_equation}
 \varrho_{\rm sc}(x)=\bm1_{x\in[-2,2]}\frac{\sqrt{4-x^2}}{2\pi},
 \quad 
 \msc(z)=\int_\bR \frac{\varrho_{\rm sc}(x)\rd x}{x-z}=\frac{-z+\sqrt{z^2-4}}{2}.
\end{split}\end{align}

  Let $\wt H$ be another $N\times N$ Wigner matrix, obtained by resampling the rows and columns of $H$ indexed by $\bS$. Specifically, the rows and columns indexed by $\bS$ are replaced with $\{\widetilde H_{ix}\}_{i\in \bS, x\in \qq{N}}$, which is an independent copy of $\{H_{ix}\}_{i\in \bS, x\in \qq{N}}$. We denote the Green's function of $\wt H$ by $\widetilde G(z)$. By construction $H$ and $\widetilde H$ have the same law, and consequently, $ G(z)$ and $\widetilde G(z)$ also share the same law.
Let $H^{(\bS)}$ and $\wt H^{(\bS)}$ denote the matrices obtained from $H$ and $\wt H$ by removing the rows and columns indexed by $\bS$. By definition, $H^{(\bS)}=\wt H^{(\bS)}$, and their Green's functions satisfy $G^{(\bS)}(z)=\wt G^{(\bS)}(z)$.

Abbreviate $\mathbb S^\complement\deq \qq{N}\backslash\mathbb S$. Let $z=2+N^{-2/3}w=2+\OO(N^{-2/3})$, we can study the edge scaling limit of $\widetilde G|_{\bS}(z)$ using the Schur complement formula \eqref{e:Schur1} 
\begin{align}\label{e:HSchur}
\wt G|_{\bS}(z)
=\frac{1}{\wt H|_\bS-z-\wt H_{\bS \bS^\complement} G^{(\bS)}(z) \wt H_{\bS^\complement \bS}}
=\frac{1}{-2-\msc(2) -\cD(z)} =\frac{1}{1/\msc(2)  -\cD(z)} ,
\end{align}
where we used that $\msc^2(z)+z\msc(z)+1=0$ and for $i,j\in \bS$, 
\begin{align}\label{e:defD}
\cD_{ij}(z)=(z-2)\delta_{ij}-\wt H_{ij}+\sum_{x,y\not\in \bS}\wt H_{ix} G^{(\bS)}_{xy}(z) \wt H_{yj}-\msc(2)\delta_{ij}.
\end{align}

To establish the weak convergence of the Green's functions $\wt G|_{\bS}(z)$, we rely on the following estimates, which are straightforward consequences of optimal rigidity for Wigner matrices (see \cite{erdHos2017dynamical,erdHos2012rigidity}). For the purpose of this discussion, we assume these results as given.
For any arbitrarily small $\fo>0$ and large $\fC>0$, the following bounds hold with probability at least $1-N^{-\fC}$ provided $N$ is large enough,
\begin{align}\label{e:input}
|\cD_{ij}(z)|\lesssim  N^{-1/3+\fo}, \text{ for } i,j\in \bS, \quad \Im[G_{xx}(z)]\lesssim N^{-1/3+\fo},\text{ for } x\in \qq{N}.
\end{align}

 For the entry-wise imaginary part of a matrix, we have the following identity
    \begin{align}\label{e:Ainverse}
        \Im[A^{-1}]= -A^{-1}\Im[A] \overline{A^{-1}},
    \end{align}
    where for any matrix $B$, $\Im[B]$ denotes the entry-wise imaginary part of $B$, and $\overline{B}$ denotes the entry-wise complex conjugate of $B$.    Thus, by taking imaginary part of both sides of \eqref{e:HSchur}, we get
\begin{align}\label{e:G-msc}\begin{split}
\Im [ \wt G|_{\bS}(z) ]
&= - (1/\msc(2) -\cD(z))^{-1}\Im[1/\msc(2) -\cD(z)] \overline{(1/\msc(2) -\cD(z))^{-1}} \\
&= (1/\msc(2) -\cD(z))^{-1}\Im[\cD(z)]  \overline{(1/\msc(2) -\cD(z))^{-1}} \\
&=(\msc(2)+\OO(|\cD(z)|))\Im[\cD(z)](\msc(2)+\OO(|\cD(z)|))=\Im[\cD(z)]  +\OO(N^{-2/3+2\fo})
	\end{split}
\end{align}
with overwhelmingly high probability. Here in the second statement we used that $\msc(2)=-1$ is real; in the last two statements we used Taylor expansion and \eqref{e:input}.
By rearranging, we conclude that

For the imaginary of $\cD(z)$ (recall from \eqref{e:defD}), we can rewrite it as
\begin{align}\label{e:Greplace_error1}
\sum_{x,y\not\in \bS}\wt H_{ix} G^{(\bS)}_{xy}(z) \wt H_{yj}=\sum_{x,y\not\in \bS}\wt H_{ix}  G_{xy}(z)\wt H_{yj}
+\sum_{x,y\not\in \bS}\wt H_{ix} (G^{(\bS)}(z)- G(z))_{xy} \wt H_{yj}.
\end{align}
It turns out the second term on the RHS of \eqref{e:Greplace_error1} is small. In fact, by the Schur complement formula we have
\begin{align}\label{e:schur_copy}
G^{(\bS)}
=  G|_{\bS^\complement}- G|_{\bS^\complement\bS}( G|_{\bS})^{-1} G|_{\bS\bS^\complement}.
\end{align}
We remark that the RHS of the above Schur complement formula also extends to $0$ for the rows and columns indexed by $\bS$, i.e. for $i\in \bS$ and $j\in \qq{N}$,
\begin{align*}
 G_{ij}-\sum_{x,y\in \bS} G_{ix}( G|_{\bS})_{xy}^{-1} G_{yj}
= G_{ij}-\sum_{y\in \bS}\delta_{iy} G_{yj}=0.
\end{align*}

By plugging \eqref{e:schur_copy} into the second term of \eqref{e:Greplace_error1}, we conclude
\begin{align}\label{e:Greplace_error2}
\sum_{x,y\not\in \bS}\wt H_{ix} (G^{(\bS)}(z)-G(z))_{xy}\wt H_{yj}
=\sum_{x,y\in \qq{N}}\wt H_{ix} (G(z) ( G|_\bS(z))^{-1} G(z))_{xy}\wt  H_{yj}
\lesssim N^{-2/3+2\fo},
\end{align}
where we used that $\wt H_{ix}$ and $G(z)$ are independent for any $x\in \qq{N}$. For any large $\fC>0$, the following bounds hold with probability at least $1-N^{-\fC}$ provided $N$ is large enough,
\begin{align*}
\sum_{x\in \qq{N}}\wt H_{ix} G_{xk}(z)\lesssim N^{\fo/2} \sqrt{\frac{\sum_{x\in \qq{N}}|G_{xk}^2(z)|}{N}}
=N^{\fo/2} \sqrt{\frac{\Im[G_{kk}(z)]}{N\Im[z] }}\lesssim N^{-1/3+\fo},
\end{align*}
where in the first statement we used concentration of measure; in the second statement we used the Ward identity \eqref{e:Ward}; in the last statement we used \eqref{e:input}.
Combining \eqref{e:G-msc}, \eqref{e:Greplace_error1} and \eqref{e:Greplace_error2}
\begin{align*}
\Im [\wt G|_{\bS}(z)]=(\wt H\Im[ G(z)] \wt H)|_\bS  +\OO(N^{-2/3+2\fo}+\Im[z]).
\end{align*}
In particular for any $i,j\in \bS$, and $z=2+N^{-2/3}w$,
\begin{align}\begin{split}\label{e:prelimit}
N^{1/3}\Im[G_{ij}(z)]
&=N^{1/3}\Im[(\wt HG(z) \wt H)_{ij}]+\OO(N^{-2/3+2\fo})\\
&=N^{1/3}\Im \sum_{s=1}^N\frac{\langle \bmu_s, \wt H_{\cdot i}\rangle\langle \bmu_s, \wt H_{\cdot j}\rangle}{\lambda_s -z}+\OO(N^{-1/3+2\fo})\\
&=\Im \sum_{s=1}^N\frac{N\langle \bmu_s, \wt H_{\cdot i}\rangle\langle \bmu_s, \wt H_{\cdot j}\rangle}{N^{2/3}(\lambda_s-2) -N^{2/3}(z-2)}+\OO(N^{-1/3+2\fo})\\
&=\sum_{s=1}^N \frac{\Im[w]N\langle \bmu_s, \wt H_{\cdot i}\rangle\langle \bmu_s, \wt H_{\cdot j}\rangle}{ (N^{2/3}(\lambda_s-2)- \Re[w])^2+\Im[w]^2}+\OO(N^{-1/3+2\fo}).
\end{split}\end{align}
We notice that $\wt H_{\cdot i}$ and the eigenvectors $\bmu_s$ are independent. Together with the delocalization of eigenvectors, it follows that $\sqrt N\langle \bmu_s, \wt H_{\cdot i}\rangle$ weakly converges to independent Gaussian random variables $Z_s(i)$.
The claim \eqref{e:edgeG_converge} then follows from \eqref{e:prelimit}, the edge universality \eqref{e:edge_eig}, and certain estimates on the locations eigenvalues (see \Cref{p:eig_concentration}). We refer to \Cref{sec:schurlemmas} for a detailed discussion in the setting of random $d$-regular graphs.

\subsection{Outline of the Paper}

In \Cref{s:preliminary}, we review key results from \cite{huang2024spectrum} on local resampling and the estimation of the Green's function for random $d$-regular graphs. We also recall the optimal eigenvalue rigidity and edge universality results from \cite{huang2024ramanujan}. Our main result, \Cref{t:ev_universality}, is established in \Cref{s:Gaussianwave}, while the proofs of certain eigenvalue estimates and the tightness of the Stieltjes transform are deferred to \Cref{s:edge_tightness}.

\subsection{Notation}
We reserve letters in mathfrak mode, e.g. $\fb,  \fc,\fo,\cdots$, to represent universal constants, and $\fC$ for a large universal constant, which
may be different from line by line. We use letters in mathbb mode, e.g. $ \bT, \mathbb X$, to represent set of vertices. 
Given a graph $\cH$, we denote the graph distance as $\dist_\cH(\cdot, \cdot)$, and the radius-$r$ neighborhood of a vertex $i\in \cH$ as $\cB_r(i,\cH)$. Given a vertex set $\mathbb X$ of $\cH$, let $\cH^{(\bX)}$ denote the graph obtained from 
$\cH$ by removing the vertices in $\bX$. Given any matrix $A$ and index sets $\bX, \mathbb Y$, let $A^{(\bX)}$ denote the matrix obtained from $A$ by removing the rows and columns indexed by $\bX$. The restriction of 
$A$ to the submatrices corresponding to $\bX\times \bX, \bX\times \mathbb Y$ are denoted as $A_\bX, A_{\bX \mathbb Y}$ respectively. For any matrix $A$, $\Im[A]$ denotes the entry-wise imaginary part of $A$, and $\overline{A}$ denotes the entry-wise complex conjugate of $A$.  

For two quantities $X_N$ and $Y_N$ depending on $N$, 
we write that $X_N = \OO(Y_N )$ or $X_N\lesssim Y_N$ if there exists some universal constant such
that $|X_N| \leq \fC Y_N$ . We write $X_N = \oo(Y_N )$, if the ratio $|X_N|/Y_N\rightarrow 0$ as $N$ goes to infinity. We write
$X_N\asymp Y_N$ if there exists a universal constant $\fC>0$ such that $ Y_N/\fC \leq |X_N| \leq  \fC Y_N$. We remark that the implicit constants may depend on $d$. 
With a slight abuse of notation, we write $\fa\ll \fb$ to indicate that $\fa/\fb\leq 0.01$. 
We denote $\qq{a,b} = [a,b]\cap\bZ$ and $\qq{N} = \qq{1,N}$. We say an event $\Omega$ holds with high probability if $\bP(\Omega)\geq 1-\oo(1)$. We say an event $\Omega$ holds with overwhelmingly high probability, if for any $\fC>0$, 
$\bP(\Omega)\geq 1-N^{-\fC}$ holds provided $N$ is large enough. 
Given an event $\Omega$, we say conditioned on $\Omega$, an event $\Omega'$ holds with overwhelmingly high probability, if for any $\fC>0$, 
$\bP(\Omega'|\Omega)\geq 1-N^{-\fC}$ holds provided $N$ is large enough.

 \subsection*{Acknowledgements.}
 The research of Y.H. is supported by National Key R\&D Program of China No.\,2023YFA1010400,  NSFC No.\,12322121 and Hong Kong RGC Grant No.\,21300223. The research of J.H. is supported by NSF grants DMS-2331096 and DMS-2337795,  and the Sloan Research Award. 
The research of H-T.Y. is supported by NSF grants DMS-1855509 and DMS-2153335. J.H. wants to thank Charles Bordenave for helpful discussions on Gaussian waves.

\section{Preliminaries}\label{s:preliminary}

In this paper we fix the parameters as follows
\begin{align}\label{e:parameters}
0<\fo\ll \ell/\log_{d-1}N\ll \fb\ll  \fc\ll \fg\ll 1,
\end{align} 
and set $\fR=(\fc/4)\log_{d-1}N$.
Below, we describe their meanings and where they are introduced:
\begin{itemize}
    \item $\fo$ arises from the delocalization of eigenvectors \eqref{e:delocalization}. Many estimates involve bounds containing $N^\fo$ factors, which are harmless.
    \item $\fb$ relates to the concentration of Green's function entries, with errors bounded by $N^{-\fb}$, see \eqref{eq:infbound0}.
    \item For the spectral parameter $z\in \bC^+$ in Green's functions and Stieltjes transforms, we restrict it to $\Im[z]\geq N^{-1+\fg}$ and $|z-2|\leq N^{-\fg}$.
    \item $\ell$ comes from local resampling in \Cref{s:local_resampling}, we resample  boundary edges of balls with radius $\ell$.
    \item $\fc$ defines $\fR$, and with high probability, random $d$-regular graphs are tree-like within radius $\fR$ neighborhoods, see \Cref{def:omegabar}.
\end{itemize}

\subsection{Kesten-Mckay Distribution}

By local weak convergence, the empirical eigenvalue density of random $d$-regular graphs converges to that of the infinite $d$-regular tree, which is known as the Kesten-McKay distribution; see \cite{kesten1959symmetric,mckay1981expected}. This density is given by 
\begin{align}\label{e:KMdistribution}
\varrho_d(x):=\mathbf1_{x\in [-2,2]} \left(1+\frac1{d-1}-\frac {x^2}d\right)^{-1}\frac{\sqrt{4-x^2}}{2\pi}.
\end{align}
We denote by $\md(z)$ the Stieltjes transform of the Kesten--McKay distribution $\varrho_d(x)$,
\begin{align*}
    \md(z)=\int_\bR \frac{\varrho_d(x)\rd x}{x-z},\quad z\in \bC^+:=\{w\in \bC \col \Im[w]>0\}.
\end{align*}
Explicitly, the Stieltjes transform of the Kesten--McKay distribution $\md(z)$ can be expressed in terms of the Stieltjes transform $\msc(z)$:
\begin{align}\label{e:md_equation}
    \md(z)=\frac{1}{-z-\frac{d}{d-1}\msc(z)},\quad \md(z)=(d-1)\frac{-(d-2)z+ d\sqrt{z^2-4}}{2(d^2-(d-1)z^2)}.
\end{align}
Note that close to the spectral edge $2$, the Stieltjes transforms $\msc(z)$ and $\md(z)$ have square root behavior: let $\cA={d(d-1)}/{(d-2)^2}$ as $z\rightarrow 2$,
\begin{align}\label{e:mscmd}
\msc(z)=-1+\sqrt{z-2}+\OO(|z-2|), \quad \md(z)=-\frac{d-1}{d-2}+\cA\sqrt{z-2}+\OO(|z-2|).
\end{align}
Let $\eta=\Im[z]$ and $\kappa=||\Re[z]|-2|$, then the imaginary part of $\msc(z)$ and $\md(z)$ satisfies (see \cite[Lemma 6.2]{erdHos2017dynamical}), 
\begin{align}\label{e:imm_behavior}
\Im[\msc(z)]\asymp \Im[\md(z)]\asymp \left\{
\begin{array}{cc}
\sqrt{\kappa+\eta}& \text{ if } |\Re[z]|\leq 2,\\
\eta/\sqrt{\kappa+\eta}&\text{ if }|\Re[z]|\geq 2. 
\end{array}
\right.
\end{align}

\subsection{Local Resampling}  
\label{s:local_resampling}

In this section, we recall the local resampling and its properties. This gives us the framework to talk about resampling from the random regular graph distribution as a way to get an improvement in our estimates of the Green's function.

For any graph $\cG$, we denote the set of unoriented edges by $E(\cG)$,
and the set of oriented edges by $\vec{E}(\cG):=\{(u,v),(v,u):\{u,v\}\in E(\cG)\}$.
For a subset $\vec{S}\subset \vec{E}(\cG)$, we denote by $S$ the set of corresponding non-oriented edges.
For a subset $S\subset E(\cG)$ of edges we denote by $[S] \subset \qq{N}$ the set of vertices incident to any edge in $S$.
Moreover, for a subset $\bV\subset\qq{N}$ of vertices, we define $E(\cG)|_{\bV}$ to be the subgraph of $\cG$ induced by $\bV$.

\begin{definition}
A (simple) switching is encoded by two oriented edges $\vec S=\{(v_1, v_2), (v_3, v_4)\} \subset \vec{E}$.
We assume that the two edges are disjoint, i.e.\ that $|\{v_1,v_2,v_3,v_4\}|=4$.
Then the switching consists of
replacing the edges $\{v_1,v_2\}, \{v_3,v_4\}$ with the edges $\{v_1,v_4\},\{v_2,v_3\}$.
We denote the graph after the switching $\vec S$ by $T_{\vec S}(\cG)$,
and the new edges $\vec S' = \{(v_1,v_4), (v_2,v_3)\}$ by
$
  T(\vec S) = \vec S'
$.
\end{definition}

The local resampling involves a fixed center vertex, which we now assume to be vertex $o$,
and a radius $\ell$.
Given a $d$-regular graph $\cG$, we write $\cT\deq\cB_{\ell}(o,\cG)$ to denote the radius-$\ell$ neighborhood of $o$ (which may not necessarily be a tree) and write $\bT$ for its vertex set.\index{$\cT, \bT$}  In addition, for any vertex set $\bX\subset \qq{N}$, and integer $r\geq 1$, we denote
	$\cB_r(\bX,\cG)=\{i\in \qq{N}: \dist_\cG(i, \bX)\leq r\}$ the ball of radius-$r$ around vertices $\bX$ in $\GG$.
The edge boundary $\del_E \cT$ of $\cT$ consists of the edges in $\cG$ with one vertex in $\bT$ and the other vertex in $\qq{N}\setminus\bT$.
We enumerate the edges of $\del_E \cT$ as $ \del_E \cT = \{e_1,e_2,\dots, e_\mu\}$, where $e_\al=\{l_\al, a_\al\}$ with $l_\al \in \bT$ and $a_\al \in \qq{N} \setminus \bT$. We orient the edges $e_\al$ by defining $\vec{e}_\al=(l_\al, a_\al)$.
We notice that $\mu$ and the edges $e_1,e_2, \dots, e_\mu$ depend on $\cG$. The edges $e_\al$ are distinct, but
the vertices $a_\al$ are not necessarily distinct and neither are the vertices $l_\al$. Our local resampling switches the edge boundary of $\cT$ with randomly chosen edges in $\cG^{(\bT)}$
if the switching is admissible (see \eqref{Wdef} below), and leaves them in place otherwise.
To perform our local resampling, we choose $(b_1,c_1), \dots, (b_\mu,c_\mu)$ to be independent, uniformly chosen oriented edges from the graph $\cG^{(\bT)}$, i.e.,
the oriented edges of $\cG$ that are not incident to $\bT$,
and define 
\begin{equation}\label{e:defSa}
  \vec{S}_\al= \{\vec{e}_\al, (b_\al,c_\al)\},
  \qquad
  {\bf S}=(\vec S_1, \vec S_2,\dots, \vec S_\mu).
\end{equation}
The sets $\bf S$ will be called the \emph{resampling data} for $\cG$. We remark that repetitions are allowed in the data $(b_1, c_1), (b_2, c_2),\cdots, (b_\mu, c_\mu)$.
We define an indicator that will be crucial to the definition of the local resampling.

\begin{definition}
For $\al\in\qq{\mu}$,
we define the indicator functions
$I_\al \equiv I_\al(\cG,{\bf S})=1$\index{$I_\alpha$} 
\begin{enumerate}
\item
 the subgraph $\cB_{\fR/4}(\{a_\al, b_\al, c_\al\}, \cG^{(\bT)})$ after adding the edge $\{a_\al, b_\al\}$ is a tree;
\item 
and $\dist_{\cG^{(\bT)}}(\{a_\al,b_\al,c_\al\}, \{a_\beta,b_\beta,c_\beta\})> {\fR/4}$ for all $\beta\in \qq{\mu}\setminus \{\al\}$.
\end{enumerate}
\end{definition}
 The indicator function $I_\alpha$ imposes two conditions. The first one is a ``tree" condition, which ensures that 
 $a_\al$ and $\{b_\al, c_\al\}$ are far away from each other, and their neighborhoods are trees. 
The second one imposes an ``isolation" condition, which ensures that we only perform simple switching when the switching pair is far away from other switching pairs. In this way, we do not need to keep track of the interaction between different simple switchings. 

We define the \emph{admissible set}
\begin{align}\label{Wdef}
{\mathsf W}_{\bf S}:=\{\al\in \qq{\mu}: I_\al(\cG,{\bf S}) \}.
\end{align}
We say that the index $\al \in \qq{\mu}$ is \emph{switchable} if $\al\in {\mathsf W}_{\bf S}$. We denote the set $\bW_{\bf S}=\{b_\al:\al\in {\mathsf W}_{\bf S}\}$\index{$\bW_{\bf S}$}. Let $\nu:=|{\mathsf W}_{\bf S}|$ be the number of admissible switchings and $\al_1,\al_2,\dots, \al_{\nu}$
be an arbitrary enumeration of ${\mathsf W}_{\bf S}$.
Then we define the switched graph by
\begin{equation} \label{e:Tdef1}
T_{\bf S}(\cG) := \left(T_{\vec S_{\al_1}}\circ \cdots \circ T_{\vec S_{\al_\nu}}\right)(\cG),
\end{equation}
and the resampling data by
\begin{equation} \label{e:Tdef2}
  T({\bf S}) := (T_1(\vec{S}_1), \dots, T_\mu(\vec{S}_\mu)),
  \quad
  T_\al(\vec{S}_\al) \deq
  \begin{cases}
    T(\vec{S}_\al) & (\al \in {\mathsf W}_{\bf S}),\\
    \vec{S}_\al & (\al \not\in {\mathsf W}_{\bf S}).
  \end{cases}
\end{equation}

To make the structure more clear, we introduce an enlarged probability space.
Equivalent to the definition above, the sets $\vec{S}_\al$ as defined in \eqref{e:defSa} are uniformly distributed over 
\begin{align*}
{\sf S}_{\al}(\cG)=\{\vec S\subset \vec{E}: \vec S=\{\vec e_\al, \vec e\}, \text{$\vec{e}$ is not incident to $\cT$}\},
\end{align*}
i.e., the set of pairs of oriented edges in $\vec{E}$ containing $\vec{e}_\al$ and another oriented edge in $\cG^{(\bT)}$.
Therefore ${\bf S}=(\vec S_1,\vec S_2,\dots, \vec S_\mu)$ is uniformly distributed over the set
${\sf S}(\cG)=\sf S_1(\cG)\times \cdots \times \sf S_\mu(\cG)$.

We introduce the following notation on the probability and expectation with respect to the randomness of the $\bfS\in \sf S(\cG)$.
\begin{definition}\label{def:PS}
    Given any $d$-regular graph $\cG$, we 
    denote $\bP_\bfS(\cdot)$ the uniform probability measure on ${\sf S}(\cG)$;
 and $\bE_\bfS[\cdot]$ the expectation  over the choice of $\bfS$ according to $\bP_\bfS$. 
\end{definition}

 The following claim from \cite[Lemma 7.3]{huang2024spectrum} states that local resampling is invariant under the random regular graph distribution.

\begin{lemma}[{\cite[Lemma 7.3]{huang2024spectrum}}] \label{lem:exchangeablepair}
Fix $d\geq 3$. We recall the operator $T_\bfS$ from \eqref{e:Tdef1}. Let $\cG$ be a random $d$-regular graph  and $\bfS$ uniformly distributed over $\sfS(\cG)$, then the graph pair $(\cG, T_{\bf S}(\cG))$ forms an exchangeable pair:
\begin{align*}
(\cG, T_{\bf S}(\cG))\stackrel{law}{=}(T_{\bf S}(\cG), \cG).
\end{align*}
\end{lemma}

In the following, we introduce the set $\oOmega$ of $d$-regular graphs with $N$ vertices, which are locally tree-like. Estimates for the Green's function and Stieltjes transform (see \eqref{e:defepsilon}) hold with overwhelmingly high probability conditioned on $\oOmega$. 
\begin{definition}\label{def:omegabar}
	Fix $d\geq 3$ and a sufficiently small $0<\fc<1$, $\fR=(\fc /4)\log_{d-1}N$ as in \eqref{e:parameters}. We define the event $\oOmega$,  where the following occur: 
	\begin{enumerate}
		\item 
		The number of vertices that do not have a  tree neighborhood of radius $\fR$ is at most $N^{\fc}$.
		\item 
		The radius $\fR$ neighborhood of each vertex has an excess (i.e., the number of independent cycles) of at most $\omega_d$. 
	\end{enumerate}
\end{definition}

The event $\oOmega$ is a typical event. The following proposition from \cite[Proposition 2.1]{huang2024spectrum} states that $\oOmega$ holds with high probability. 
\begin{proposition}[{\cite[Proposition 2.1]{huang2024spectrum}}]\label{lem:omega}
	$\oOmega$ occurs with probability $1-\OO(N^{-(1-\fc)\omega_d})$.
\end{proposition}

To ensure the switching edges are sufficiently spaced apart and have large tree neighborhoods, the following indicator functions can be utilized.

\begin{definition}
    \label{def:indicator}
Fix a $d$-regular graph $\cG\in \oOmega$, and an edge $\{i,o\}$ in $\cG$.
We consider the local resampling around $(i,o)$, with resampling data $\bfS=\{(l_\al, a_\al), (b_\al, c_\al)\}_{\al\in \qq{\mu}}$. Let $\cF=(i,o)\cup \{(b_\al, c_\al)\}_{\al\in \qq{\mu}}$, and $I(\cF, \cG)=1$ to denote the indicator function on the event that vertices close to edges in $\cF$ have radius $\fR$ tree neighborhoods, and edges in $\cF$ are distance $3\fR$ away from each other. Explicitly,  $I(\cF, \cG)$ is given by
\begin{align*}
   I(\cF, \cG):= \prod_{(b,c)\in \cF}A_{bc}\prod_{(b,c)\in \cF}\left(\prod_{x\in \cB_\ell(c,\cG)}\bm1(\cB_{\fR}(x, \cG) \text{ is a tree})\right) \prod_{(b,c)\neq (b',c')\in \cF}\bm1(\dist_\cG(c,c')\geq 3\fR).
\end{align*}
\end{definition}

The following  lemma from \cite[Lemma 3.16]{huang2024ramanujan} states that with high probability with respect to the randomness of $\bfS$, the randomly selected edges $(b_\al, c_\al)$ are far away from each other, and have large tree neighborhood. In particular $\wt \cG=T_\bfS \cG\in \oOmega$.
\begin{lemma}[{\cite[Lemma 3.16]{huang2024ramanujan}}]\label{lem:configuration}
Fix a $d$-regular graph $\cG\in \oOmega$, and an edge $\{i,o\}$ in $\cG$, such that for any $x\in \cB_\ell(o,\cG)$, $\cB_\fR(x,\cG)$ is a tree.
We consider the local resampling around $(i,o)$, with resampling data $\{(l_\al, a_\al), (b_\al, c_\al)\}_{\al\in \qq{\mu}}$. We denote the set of resampling data $\sfF(\cG)\subset \sfS(\cG)$ such that the following holds 
\begin{enumerate}
    \item 
for any $\al\neq \beta\in \qq{\mu}$, $\dist_\cG(\{i,o\}\cup \{b_\al, c_\al\},\{b_\beta, c_\beta\} )\geq3\fR$; 
\item 
for any  $v\in \cB_{\ell}(\{b_\al, c_\al\}_{\al\in \qq{\mu}}, \cG)$, the radius $\fR$ neighborhood of $v$ is a tree.
\end{enumerate}
Then $\bP_\bfS(\sfF(\cG))\geq 1-N^{-1+2\fc}$ (where $\bP_{\bfS}(\cdot)$ is the probability with respect to the randomness of $\bfS$ as in \Cref{def:PS}). Also, for $\bfS\in \sfF(\cG)$ the following holds:
 $\mu=d(d-1)^\ell$, ${\mathsf W}_{\bf S}=\qq{\mu}$ (recall from \eqref{Wdef}),  $\tcG=T_\bfS(\cG)\in \oOmega$.
\end{lemma}

\subsection{Green's function extension with general weights}\label{s:pre}
Fix degree $d\geq 3$, we recall the integer $\omega_d\geq 1$ from \cite[Definition 2.6]{huang2024spectrum}, which represents the maximum number of cycles a $d$-regular graph can have while still ensuring that its Green's function exhibits exponential decay. Instead of delving into the technical definition, it suffices to note two key properties: $\omega_d\geq 1$ and $\omega_d$ is nondecreasing with respect to $d$.


As we will see, for graphs $\cG\in \oOmega$,  their Green's functions can be approximated by tree extensions with overwhelmingly high probability. 
For the infinite $d$-regular tree $\cX_d$
 and the infinite $(d-1)$-ary tree $\cY_d$ (trees where the root has degree $d-1$ and all other vertices have degree $d$),
the following proposition computes its Green's function explicitly.
\begin{proposition}[{\cite[Proposition 2.2]{huang2024spectrum}}]\label{greentree}
Let $\cX_d$ be the infinite $d$-regular tree.
For all $z \in \bC^+$, its Green's function is
\begin{equation} \label{e:Gtreemkm}
  G^{\cX_d}_{ij}(z)=m_{d}(z)\left(-\frac{\msc(z)}{\sqrt{d-1}}\right)^{\dist_{\cX_d}(i,j)}.
\end{equation}
Let $\cY_d$ be the infinite $(d-1)$-ary tree with root vertex $o$.
Its Green's function is
\begin{equation} \label{e:Gtreemsc}
  G^{\cY_d}_{ij}(z)=m_{d}(z)\left(1-\left(-\frac{\msc(z)}{\sqrt{d-1}}\right)^{2{\rm anc}(i,j)+2}\right)\left(-\frac{\msc(z)}{\sqrt{d-1}}\right)^{\dist_{\cY_d}(i,j)},
\end{equation}
where ${\rm anc}(i,j)$ is the distance from the common ancestor of the vertices $i,j$ to the root $o$. 
In particular,
\begin{align}\label{e:Gtreemsc2}
G^{\cY_d}_{oi}(z)=\msc(z)\left(-\frac{\msc(z)}{\sqrt{d-1}}\right)^{\dist_{\cY_d}(o,i)}.
\end{align}
\end{proposition}

Next, we recall the idea of a Green's function extension with general weight $\Delta\in \mathbb C$ from {\cite[Section 2.3]{huang2024spectrum}}.

\begin{definition}\label{def:pdef}
    Fix degree $d\geq 3$, and a graph $\cT$ with degrees bounded by $d$. We define the function $P(\cT,z,\Delta)$ as follows. We denote $A(\cT)$  the adjacency matrix of $\cT$, and $D(\cT)$ the diagonal matrix of degrees of $\cT$. Then 
    \begin{align}\label{e:defP}
    P(\cT,z,\Delta):=\frac{1}{-z+A(\cT)/\sqrt{d-1}-(d -D(\cT))\Delta/(d-1)}.
    \end{align}
\end{definition}
The matrix $P(\cT,z,\Delta)$ is the Green's function of the matrix obtained from $A(\cT)/\sqrt{d-1}$ by attaching to each vertex $i\in \cT$ a weight $-(d-D_{ii}(\cT))\Delta/(d-1)$. When $\Delta=\msc(z)$, \eqref{e:defP} is the Green's function of the tree extension of $\cT$, i.e. extending $\cT$ by attaching copies of infinite $(d-1)$-ary trees to $\cT$ to make each vertex degree $d$. If $\cT$ is a tree, then in this case, the Green's function agrees with the Green's function of the infinite $d$-regular tree, as in \eqref{e:Gtreemkm}. 
For any vertex set $\bX$ in $\cT$, we define the following Green's function with vertex $\bX$ removed. Let $A^{(\bX)}(\cT)$ and $D^{(\bX)}(\cT)$ denote the matrices obtained from $A(\cT), D(\cT)$ by removing the row and column associated with vertex $\bX$. The Green's function is then defined as:
 \begin{align}\label{e:defPi}
    P^{(\bX)}(\cT,z,\Delta):=\frac{1}{-z+A^{(\bX)}(\cT)/\sqrt{d-1}-(d-D^{(\bX)}(\cT))\Delta/(d-1)}.
    \end{align}

%
%

For any integer $\ell\geq 1$, we define the functions $X_\ell(\Delta,z), Y_\ell(\Delta,z)$ as
\begin{align}\label{def:Y}
X_\ell(\Delta,z)=P_{oo}(\cB_\ell(o,\cX_d),z,\Delta),\quad Y_\ell(\Delta,z)=P_{oo}(\cB_\ell(o,\cY_d),z,\Delta),
\end{align}
where $\cX_d$ is the infinite $d$-regular tree with root vertex $o$, and $\cY_d$ is the infinite $(d-1)$-ary tree with root vertex $o$. 
Then $\msc(z)$ is a fixed point of the function $Y_\ell$, i.e. $Y_\ell(\msc(z),z)=\msc(z)$. Also, $X_\ell(\msc(z),z)=\md(z)$. 
The following proposition states that if $\Delta$ is sufficiently close to $\msc(z)$ and $w$ is close to $z$, then $Y_\ell(\Delta,w)$ is close to $\msc(z)$, and $X_\ell(\Delta,w)$ is close to $\md(z)$. We postpone its proof to \Cref{app:Green}.

\begin{proposition}\label{p:recurbound}
Given $z, \Delta\in \bC^+$ such that $\ell|\Delta-\msc(z)|\ll 1$, then the derivatives of $Y_\ell(\Delta,w)$ satisfies
\begin{align}\label{e:Yl_derivative}
   |\del_1 Y_\ell(\Delta,  z)| \lesssim 1, \quad 
   |\del_1^2 Y_\ell(\Delta,  w)|\lesssim \ell,
\end{align}
and the same statement holds for $X_\ell(\Delta,w)$.
Moreover, the functions $X_\ell(\Delta,w), Y_\ell(\Delta,w)$ 
satisfy
\begin{align}\begin{split}\label{e:recurbound}
&\phantom{{}={}}Y_\ell(\Delta,z)-\msc(z)
=\msc^{2\ell+2}(z)(\Delta-\msc(z))\\
&+\msc^{2\ell+2}(z)\md(z)\left(\frac{1-\msc^{2\ell+2}(z)}{d-1}+\frac{d-2}{d-1}\frac{1-\msc^{2\ell+2}(z)}{1-\msc^2(z)}\right)(\Delta-\msc(z))^2+\OO(\ell^2|\Delta-\msc(z)|^3).
\end{split}\end{align}
and
\begin{align}\begin{split}\label{e:Xrecurbound}
X_\ell(\Delta,z)-\md(z)
&=\frac{d}{d-1}\md^2(z)\msc^{2\ell}(z)(\Delta-\msc(z))+\OO\left(\ell|\Delta-\msc(z)|^2\right).
\end{split}\end{align}
\end{proposition}

\subsection{Green's function Estimates}

The weak local law of $H=H(0)$ has been proven in {\cite[Theorem 4.2]{huang2024spectrum}}, which is recalled below (by taking $(\fa, \fb,\fc, \frak r)=(12,300,\fc,\fc/32)$). We denote 
\begin{align}\label{e:defQ}
G(z)=(H-z)^{-1}, \quad m_N(z)=\frac{1}{N}\Tr[G(z)], \quad Q(z):=\frac{1}{dN}\sum_{i\sim j}G_{ii}^{(j)}(z),
\end{align}
where the summation $i\sim j$ is over all the (directed) edges $(i,j)$ of $\cG$. 

\begin{theorem}[{\cite[Theorem 4.2]{huang2024spectrum}}] \label{thm:prevthm0}
Fix any sufficiently small $0<\fb\ll\fc<1$, $\fR=(\fc/4)\log_{d-1}(N)$ and any $z\in \bC^+$ , we define $\eta(z)=\Im[z], \kappa(z)=\min\{|\Re[z]-2|, |\Re[z]+2|\}$, and the error parameters
\begin{align}\label{e:defepsilon}
\varepsilon'(z):=(\log N)^{100}\left(N^{-10\fb} +\sqrt{\frac{\Im[m_d(z)]}{N\eta(z)}}+\frac{1}{(N\eta(z))^{2/3}}\right),\quad \varepsilon(z):=\frac{\varepsilon'(z)}{\sqrt{\kappa(z)+\eta(z)+\varepsilon'(z)}}.
\end{align}
We recall $P$ from \eqref{e:defP}.
For any $\fC\geq 1$ and $N$ large enough, with probability at least $1-\OO(N^{-\fC})$ with respect to the uniform measure on $\oOmega$, 
\be\label{eq:infbound0}
|G_{ij}(z)-P_{ij}(\cB_{\fR/100}(\{i,j\},\cG),z,\msc(z))|,\quad |Q(z)-\msc(z)|,\quad |m_N(z)-m_d(z)|\lesssim \varepsilon(z).
\ee
uniformly for every $i,j\in \qq{N}$, and any $z\in \bC^+$ with $\Im[z]\geq (\log N)^{300}/N$. We denote the event that \eqref{eq:infbound0} holds as $\Omega$. 
\end{theorem}

We notice that the first statement in \eqref{eq:infbound0} implies that for $\Im[z]\geq (\log N)^{300}/N$, $\Im[G_{ii}(z)]\lesssim 1$. It follows that eigenvectors are delocalized 
\begin{align}\label{e:delocalization}
    \|\bmu_s\|^2_\infty\lesssim \max_{1\leq i\leq N}((\log N)^{300}/N)\Im[G_{ii}(\lambda_s+\ri (\log N)^{300}/N)]\lesssim (\log N)^{300}/N\ll N^{\fo/2-1}.
\end{align} 

In the rest of this section we collect some basic estimates of the Green's function $G(z)$, and the Stieltjes transform $m_N(z)$ and $\wt m_N(z)$ (the Stieltjes transform of the switched graph) from \cite[Lemma 2.18 and Lemma 2.19]{huang2024ramanujan}.  By taking $t=0$, in \cite[Lemma 2.18 and Lemma 2.19]{huang2024ramanujan}, then $G(z,0)=G(z)$, $m_0(z)=m_N(z)$, and $\wt m_0(z)=\wt m_N(z)$.

\begin{lemma}[{\cite[Lemma 2.18 and Lemma 2.19]{huang2024ramanujan}}]\label{l:basicG}
Fix a $d$-regular graph $\cG\in \Omega$, and an edge $\{i,o\}$ in $\cG$.
We consider the local resampling around $(i,o)$, with resampling data $\bfS=\{(l_\al, a_\al), (b_\al, c_\al)\}_{\al\in \qq{\mu}}$. We denote $\cT=\cB_\ell(o,\cG)$ with vertex set $\bT$, and the switched graph $\wt \cG=T_\bfS(\cG)$. We let $\cF=(i,o)\cup \{(b_\al, c_\al)\}_{\al\in \qq{\mu}}$, and assume that $\cG, \wt \cG\in \Omega$, and 
$
I(\cF, \cG)=1$ (recall from \Cref{def:indicator}). The following holds: 
   \begin{align}\label{e:tmmdiff}
    |\wt m_N (z)-m_N (z)|\lesssim \frac{(d-1)^\ell N^\fo \Im[m_N (z)]}{N\eta}.
\end{align}
We recall $P^{(\bX)}$ from \eqref{e:defPi}.   For any vertex set $\bX=\emptyset$, $\{b_\al b_\beta\}$, $\bT$, $\bT\cup \{b_\al, b_\beta\}$, or $\bT\cup \bW$:
\begin{align}\label{eq:local_law}
&\max_{x,y\not\in \bX}|G^{(\bX)}_{xy}(z)-P^{(\bX)}_{xy}(\cB_{\fR/100}(\{x,y\}\cup \bX, \GG),z ,\msc(z ))|\lesssim N^{-\fb},\\
\label{e:Gest}
        &\max_{x,y\not\in \bX}|\Im[ G^{(\bX)}_{xy}(z)]|\lesssim N^\fo \Im[ m_N (z)]
    \end{align}
    
\end{lemma}

%
%
\subsection{Edge Universality and Optimal Rigidity}

In this section we recall the main results on the optimal rigidity and edge universality of random $d$-regular graphs from  \cite{huang2024ramanujan}.

For $2\leq i \leq N$, we expect the $i$-th largest eigenvalue of the normalized adjacency matrix $H$ to be closed to the classical eigenvalue locations $\gamma_i$, where $\gamma_i$ is defined so that 
\be\label{eq:gammadef}
\int_{\gamma_i}^2 \varrho_d(x)\rd x=\frac{i-1/2}{N-1},\quad 2\leq i\leq N.
\ee

\begin{theorem}[{\cite[Theorem 1.1]{huang2024ramanujan}}]\label{thm:eigrigidity}
Fix $d\geq 3$, and an arbitrarily small $\fa>0$. There exists a positive integer $\omega_d \geq 1$, depending only on $d$, such that conditioned on $\cG\in \Omega$, with overwhelmingly high probability, the eigenvalues $\lambda_1 = {d}/{\sqrt{d-1}} \geq \lambda_2 \geq \cdots \geq \lambda_N$ of the normalized adjacency matrix $H$ of random $d$-regular graphs satisfy:
\begin{align}\label{e:optimal_rigidity}
|\lambda_i-\gamma_i|\leq N^{-2/3+\fa}(\min\{i,N-i+1\})^{-1/3},
\end{align}
for every $2\leq i\leq N$ and $\gamma_i$ are the classical eigenvalue locations, as defined in \eqref{eq:gammadef}.
\end{theorem}

\begin{theorem}[{\cite[Theorem 1.2]{huang2024ramanujan}}]
\label{t:universality}
Fix $d\geq 3$, $k\geq 1$ and $s_1,s_2,\cdots, s_k \in \bR$, and let $\cA=d(d-1)/(d-2)^2$. There exists a small $\fa>0$ such that the eigenvalues $\lambda_1 = {d}/{\sqrt{d-1}} \geq \lambda_2 \geq \cdots \geq \lambda_N$ of the normalized adjacency matrix $H$ of random $d$-regular graphs satisfy: 
\begin{align*}\begin{split}
 \bP_{H}\left( (\cA N)^{2/3} ( \lambda_{i+1} - 2 )\geq s_i,1\leq i\leq k \right)= \bP_{\mathrm{GOE}}\left( N^{2/3} ( \mu_i - 2  )\geq s_i,1\leq i\leq k \right) +\OO(N^{-\fa}),
\end{split}\end{align*}
where $\mu_1\geq \mu_2\geq \cdots \geq\mu_N$ are the eigenvalues of the GOE.
The analogous statement holds for the smallest eigenvalues $-\lambda_N,\ldots,-\lambda_{N-k+1}$.
\end{theorem}

We recall the two functions $X_\ell, Y_\ell$ from \eqref{def:Y}, the Stieltjes transform $m_N(z)$ and the quantity $Q(z)$ from \eqref{e:defQ}, In the following proposition, we collect some estimates of $m_N(z)$ and $Q(z)$ from \cite[Corollary 3.6 and Corollary C.3]{huang2024ramanujan}. 

By taking $t=0$, in \cite[Corollary C.3]{huang2024ramanujan}, then $Q_0(z)=Q(z)$, $m_0(z)=m_N(z)$, $Y_0(z)=Y_\ell(Q(z),z)$, $X_0(z)=X_\ell(Q(z),z)$, and 
\begin{align*}
\Phi(z)=\frac{\Im[m_N(z)]}{N\Im[z]}+\frac{1}{N^{1-2\fc}},\quad \Upsilon(z)=|1-\del_1 Y_\ell(Q,z)|+(d-1)^{8\ell}\Phi(z),\quad
\widetilde \Upsilon(z)=|1-\del_1 Y_\ell(Q,z)|.
\end{align*}
\begin{proposition}[{\cite[Corollary 3.6 and Corollary C.3]{huang2024ramanujan}}]\label{c:rigidity}
We recall the parameters $\fo\ll \ell/\log_{d-1}N\ll \fb\ll \fc\ll \fg$ from \eqref{e:parameters}.
Conditioned on $\cG\in \Omega$ (recall from \Cref{thm:prevthm0}), with overwhelmingly high probability the following holds:
For any $z\in \bC^+$ such that $\Im[z]\geq N^{-1+\fg}$ and $|z-2|\leq N^{-\fg}$, we denote $\eta=\Im[z]$ and $\kappa=|\Re[z]-2|$, then the following holds
 \begin{align}\label{e:equation_est}
&|m_N (z)-X_\ell(Q(z), z)|\lesssim  \frac{N^{2\fc}(\kappa+\eta)^{1/4}}{N\eta}+\frac{N^{6\fc}}{(N\eta)^{3/2}},\\
\label{e:Qm_est}
&|Q (z)-\msc (z)|, |m_N(z)-\md(z)|\lesssim  \frac{N^{8\fo}}{N\eta},
\end{align}
and 
\begin{align}\label{e:QYQbound2}
    &\bE[\bm1(\GG\in \Omega)|Q(z) -Y_\ell(Q(z),z) |^{2p}]
\lesssim \bE\left[\bm1(\GG\in \Omega)\left((\ell \Phi(z))^{2p}+\left(\frac{\ell \widetilde \Upsilon(z) \Phi(z)}{N\eta}\right)^p+\left(\frac{\Upsilon(z) \Phi(z)}{(d-1)^{\ell/4}N\eta}\right)^{p}\right)\right].
\end{align}

\end{proposition}

\section{Gaussian Waves and Edge Eigenvectors}\label{s:Gaussianwave}

We prove our main result \Cref{t:ev_universality} in this section. 
Fix a $d$-regular graph $\cG\in \Omega$, and an edge $\{i,o\}$ in $\cG$. Since $\Omega \subset \oOmega$ (recall from \eqref{def:omegabar}), the number of vertices in $\cG$ that do not have a tree neighborhood of radius $\fR$ is at most $N^\fc$. By the permutation invariance of vertices, with probability at least $1-\OO(N^{-1+2\fc})$, for any vertex $v\in \cB_\ell(o, \cG)$, $\cB_\fR(v,\cG)$ is a tree.

We now consider the local resampling around $(i,o)$, with resampling data $\bfS=\{(l_\al, a_\al), (b_\al, c_\al)\}_{\al\in \qq{\mu}}$. Let $\cT=\cB_\ell(o,\cG)$ with vertex set $\bT$, and denote the switched graph by $\wt \cG=T_\bfS(\cG)$. By \Cref{lem:exchangeablepair}, $\cG$ and $\wt \cG$ form an exchangeable pair. Define $\cF=(i,o)\cup \{(b_\al, c_\al)\}_{\al\in \qq{\mu}}$. By \Cref{lem:configuration} and the discussion above, conditioned on $\cG\in \Omega$, the following holds with probability at least $1-\OO(N^{-1+2\fc})$: $I(\cF,\cG)=1$ (recall from \Cref{def:indicator}) and $\wt \cG\in \oOmega$ . Furthermore, we have
\begin{align*}
\bP(\wt \cG\not\in \Omega, \cG\in \Omega)
&=\bP(\wt \cG\not\in \oOmega, \cG\in \Omega)+\bP(\wt \cG\in \oOmega\setminus\Omega, \cG\in \Omega)\lesssim N^{-1+2\fc}+\bP(\wt \cG\in \oOmega\setminus\Omega)\lesssim N^{-1+2\fc}.
\end{align*}
Thus, we can assume that $\cG, \wt \cG\in \Omega$, and that
$
I(\cF, \cG)=1
$, which holds with probability at least $1-\OO(N^{-1+2\fc})$:
\begin{align}\label{e:theset}
\bP(\cG,\tcG\in \Omega, I(\cF,\cG)=1)=1-\OO(N^{-1+2\fc}).
\end{align}

\subsection{Green's Function of Switched Graphs}\label{sec:schurlemmas}

Assume that $I(\cF, \cG)=1$, then $\cT=\cB_\ell(o,\cG)$ is a truncated $d$-regular tree at level $\ell$. Let $\bT$ be the vertex set of $\cal T$. We denote the diagonal matrix $\bI^\partial$ indexed by $\bT\times \bT$, such that for any $x,y\in \bT$, $\bI^\del_{xy}=\delta_{xy} \bm1(\dist_\cT(o,x)=\ell)$. We recall from \eqref{e:defP} the following matrix on $\bT\times \bT$ (we notice that in \eqref{e:defP} $(d-D(\cT))/(d-1)=\bI^\partial$)
\begin{align}\label{e:copyP}
P(\cT, z, \msc(z))=\frac{1}{H_{\bT}-z-\msc(z)\bI^{\del} },
\end{align}
is explicitly given by the Green's function of the infinite $d$-regular tree \eqref{e:Gtreemkm}. 

We introduce the following matrix on $\bT\times \bT$, by setting $z=2$ in \eqref{e:copyP}
\begin{align*}
P=P(\cT, 2, \msc(2))=\frac{1}{H_{\bT}-2-\msc(2)\bI^{\del} }, 
\end{align*}
Then $P$ is a real matrix, and agrees with the Green's function on the infinite $d$-regular tree (recall from \eqref{e:Gtreemkm}) .  Explicitly, for $i,j\in \bT$, we have
\begin{align}\begin{split}\label{e:Pexp}
&P_{ij}(\cT,z,\msc(z))=\md(z)\left(-\frac{\msc(z)}{\sqrt{d-1}}\right)^{\dist_\cG(i,j)},\\
&P_{ij}=\md(2)\left(-\frac{\msc(2)}{\sqrt{d-1}}\right)^{\dist_\cG(i,j)}
=-\frac{d-1}{d-2}\left(\frac{1}{\sqrt{d-1}}\right)^{\dist_\cG(i,j)}.
\end{split}\end{align}
Using \eqref{e:mscmd}, near the spectral edge $2$,  $P$ is a good approximation of $P_{ij}(\cT, z, \msc(z))$, 
\begin{align}\label{e:PPdiff}
 |P_{ij}(\cT, z, \msc(z))-P_{ij}|\lesssim \ell\sqrt{|z-2|}.
\end{align}

The following proposition provides leading order terms for  the Green's function near the spectral edge.

\begin{proposition}\label{lem:diaglem}
Fix a $d$-regular graph $\cG\in \Omega$, a radius $r\geq 0$, and an edge $\{i,o\}$ in $\cG$.
We consider the local resampling around $(i,o)$, with resampling data $\bfS=\{(l_\al, a_\al), (b_\al, c_\al)\}_{\al\in \qq{\mu}}$. Let $\cT=\cB_\ell(o,\cG)$ with vertex set $\bT$, and denote the switched graph as $\wt \cG=T_\bfS(\cG)$. Define the set $\cF=(i,o)\cup \{(b_\al, c_\al)\}_{\al\in \qq{\mu}}$, and let $z=2+w/(\cA N)^{2/3}$. The following holds  with overwhelmingly high probability conditioned on $\cG, \wt \cG\in \Omega$, and 
$
I(\cF, \cG)=1$  (recall from \Cref{def:indicator}),
\begin{align}\begin{split}\label{e:G-Y}
  \Im[\widetilde G_{ij}(z)]
   &= \sum_{\al,\beta\in\qq{\mu}}\frac{P_{il_\al}P_{jl_\beta} }{d-1}\Im\left[  G_{c_\al c_\beta}(z)-\frac{G_{b_\al c_\beta}(z)}{\sqrt{d-1}}-\frac{G_{c_\al b_\beta}(z)}{\sqrt{d-1}} +\frac{G_{b_\al b_\beta}(z)}{d-1}\right]+\OO(N^{-\frac{1}{3}-\frac{\fb}{4}}),
\end{split}\end{align}
where the error bound is uniform for $w$ in any compact subset of $\bC^+$, and $i,j\in \bS$, the vertex set of $\cB_r(o,\cG)$.  
\end{proposition}

We denote the Green's functions of $\cG^{(\bT)}$ and $\wt\cG^{(\bT)}$ as $ G^{(\bT)}(z)$ and $\wt G^{(\bT)}(z)$,  respectively.  
Before proving \Cref{lem:diaglem}, we analyze  the error introduced by replacing $\wt G^{(\bT)}_{c_\al c_\beta}(z)$ with $ G_{c_\al c_\beta}^{(b_\al b_\beta)}(z)$. Additionally, we present a lemma, which will be used to estimate $G_{c_\al c_\beta}^{(b_\al b_\beta)}(z)$.

\begin{lemma}\label{l:diffG1}
Adopt the notation and assumptions of \Cref{lem:diaglem}.  Let $z=2+w/(\cA N)^{2/3}$.  The following holds  with overwhelmingly high probability conditioned on $\cG, \wt \cG\in \Omega$, and 
$
I(\cF, \cG)=1$  (recall from \Cref{def:indicator}),
    \begin{align*}
   &\Im[ \wt G^{(\bT)}_{c_\al c_\beta}(z)-G^{(b_\al b_\beta)}_{c_\al c_\beta}(z)]\lesssim N^{-1/3-\fb/2}.
    \end{align*}
   where the error bound is uniform for any indices $\al,\beta\in\qq{\mu}$, and $w$ in any compact subset of $\bC^+$. 
\end{lemma}

\begin{lemma}\label{l:ImGG}Adopt the notation and assumptions of \Cref{lem:diaglem}. Let $z=2+w/(\cA N)^{2/3}$.   Then for any  $(c,b)\neq (c',b')\in \cF$, the following holds  with overwhelmingly high probability conditioned on $\cG, \wt \cG\in \Omega$, and 
$
I(\cF, \cG)=1$  (recall from \Cref{def:indicator}),
\begin{align}\label{e:Gbcc}
\Im[G_{cc}^{(b)}(z)]=\Im[G_{cc}(z)]-\frac{2\Im[G_{bc}(z)]}{\sqrt{d-1}}+\frac{\Im[G_{bb}(z)]}{d-1}+\OO(N^{-1/3-\fb/2}),
\end{align}
and 
\begin{align}\begin{split}\label{e:Gbbcc}
 \Im[G_{cc'}^{(bb')}(z)]&=   \Im[G_{cc'}(z)]-\frac{\Im[G_{bc'}(z)]}{\sqrt{d-1}}-\frac{\Im[G_{cb'}(z)]}{\sqrt{d-1}} +\frac{\Im[G_{bb'}(z)]}{d-1}  +\OO(N^{-1/3-\fb/2}),
\end{split}\end{align}
where the error bound is uniform for $w$ in any compact subset of $\bC^+$. 
\end{lemma}

\begin{proof}[Proof of \Cref{lem:diaglem}]
For any $x,y\in \bT$, we recall that $\bI^\del_{xy}=\delta_{xy} \bm1(\dist_\cT(o,x)=\ell)$.
By  the Schur complement formula \eqref{e:Schur1}, we have 
\begin{align}\label{e:HSchur_graph}
\widetilde G(z)=
\frac{1}{H_{\bT}-z-\widetilde  B^\top\widetilde G^{(\bT)}(z)\widetilde  B}
=\frac{1}{H_{\bT}-2-\msc(2)\bI^{\del} -\cD(z)},
\end{align}
where for $x,y\in \bT$,
\begin{align}\label{e:Dexp}
\cD_{xy}(z)=(z-2)\delta_{xy}+\sum_{\al, \beta: l_\al =x, l_\beta=y}\frac{\widetilde G_{c_\al c_\beta}^{(\bT)}}{d-1} - \msc(2)\bI^\del_{xy}.
\end{align}

Let  $z=2+w/(\cA N)^{2/3}$, conditioned on $\cG, \wt \cG\in \Omega$ and $I(\cF,\cG)=1$, the following holds with overwhelmingly high probability
\begin{align}\begin{split}\label{e:Immbound}
&\Im[m_N(z)]=\Im[\md(z)]+\OO\left(\frac{N^{8\fo}}{N\Im[z]}\right)\lesssim \sqrt{|z-2|}+\OO\left(\frac{N^{8\fo}}{N^{1/3}\Im[w]}\right)=\OO(N^{-1/3+8\fo}),\\
&\Im[\wt m_N(z)]\lesssim \Im[m_N(z)]+\frac{(d-1)^{\ell}N^\fo\Im[m_N(z)]}{N\Im[z]}\lesssim \Im[ m_N(z)]=\OO(N^{-1/3+8\fo}),
\end{split}\end{align}
where the first line follows from \eqref{e:mscmd} and \eqref{e:Qm_est}; and the second line follows from \eqref{e:tmmdiff}.
Moreover, conditioned on $\cG, \wt \cG\in \Omega$ and $I(\cF,\cG)=1$, for $i,j\in \bT$
it follows from \eqref{eq:local_law} and \eqref{e:PPdiff}, 
\begin{align}\label{e:cDbound}
\wt G_{ij}(z) =P_{ij}+(\wt G_{ij}(z) -P_{ij}(\cT,z,\msc(z)))+(P_{ij}(\cT,z,\msc(z))-P_{ij})
=P_{ij}+\OO(N^{-\fb}),
\end{align}
and it follows from \eqref{e:Gest} and \eqref{e:Immbound}
\begin{align}\label{e:ImcDbound}
\quad \Im[\cD_{ij}]\lesssim \Im[z]\delta_{ij}+\max_{x,y\not\in \bT}\Im[\wt G^{(\bT)}_{xy}]
\lesssim N^{-1/3}+ N^\fo \Im[\wt m_N(z)]\lesssim N^{-1/3+10\fo}.
\end{align}

    Thus, by taking imaginary part of both sides of \eqref{e:HSchur_graph}, we get the following holds with overwhelmingly high probability conditioned on $\cG, \wt \cG\in \Omega$ and $I(\cF,\cG)=1$,
\begin{align}\begin{split}\label{e:ImtG}
&\phantom{{}={}}\Im [\wt G_{ij}(z) ]
= -( \wt G|_{\bT}(z) \Im[H_{\bT}-z-\widetilde  B^\top\widetilde G^{(\bT)}(z) \widetilde  B]  \overline{\wt G|_{\bT}(z) })_{ij} \\
&= - \sum_{x,y\in \bT}(P_{ix}+\OO(N^{-\fb}))\Im[(H_{\bT}-2-\msc(2)\bI^{\del} -\cD(z))_{xy}]  (P_{yj}+\OO(N^{-\fb}))\\
&= \sum_{x,y\in \bT}(P_{ix}+\OO(N^{-\fb})\Im[\cD_{xy}(z)]  (P_{yj}+\OO(N^{-\fb})) \\
&=(P\Im[\cD(z) ] P)_{ij} +\OO\left(\sum_{x,y\in \bT} N^{-\fb} N^{-1/3+10\fo} |P_{yj}|+|P_{ix}| N^{-1/3+10\fo} N^{-\fb}+N^{-2\fb} N^{-1/3+10\fo}\right)\\
&=(P\Im[\cD(z) ] P)_{ij}+\OO((d-1)^{2\ell}N^{-1/3+10\fo-\fb})=(P\Im[\cD(z) ] P)_{ij}  +\OO(N^{-1/3-\fb/2}),
\end{split}\end{align}
where in the first statement we used \eqref{e:Ainverse}; in the second statement we used \eqref{e:cDbound}; in the third statement we used that $H|_\bT-2-\msc(2)\bI^\del$ is real; in the last two statements we used $|P_{ix}|, |P_{yj}|\lesssim 1$ from \eqref{e:Pexp}.

Thus we need to estimate $\Im[\cD_{xy}(z)]$. Recall from \eqref{e:Dexp}
\begin{align}\label{e:gu0}
\Im[\cD_{xy}(z) ]=\Im[z]\delta_{xy}+\sum_{\al, \beta: l_\al =x, l_\beta=y}\frac{1}{d-1}\Im[\widetilde G_{c_\al c_\beta}^{(\bT)}].
\end{align}
We recall from \Cref{l:diffG1}, with overwhelmingly high probability conditioned on $\cG, \wt \cG\in \Omega$, and 
$
I(\cF, \cG)=1$, the following holds
\begin{align}\label{e:gu1}
|\Im[\widetilde G^{(\bT)}_{c_\al c_\beta}(z) -G^{(b_\al b_\beta)}_{c_\al c_\beta}(z) ]|\lesssim N^{-1/3-\fb/2}.
\end{align}
By plugging \eqref{e:gu0} and \eqref{e:gu1} into \eqref{e:ImtG}, we get 
\begin{align}\label{e:gu2}
\Im[\wt G_{ij}(z)]=\frac{1}{d-1}\sum_{\al,\beta\in\qq{\mu}}P_{il_\al} \Im[(G_{c_\al c_\beta}^{(b_\al b_\beta)}(z) ] P_{jl_\beta}
+\OO\left(N^{-1/3-\fb/4}\right),
\end{align}
where the error term is from 
\begin{align*}
 \Im[z] |{ (P^2)_{ij}}|
+\frac{1}{d-1}\sum_{\al,\beta\in\qq{\mu}}N^{-1/3-\fb/2} |P_{il_\al } P_{jl_\beta}|\lesssim N^{-1/3-\fb/4},
\end{align*}
and we used that $(d-1)^{\ell}\leq N^{\fb/10}$.
We recall from \Cref{l:ImGG}, with overwhelmingly high probability conditioned on $\cG, \wt \cG\in \Omega$, and 
$
I(\cF, \cG)=1$, the following holds
\begin{align*}
\Im[G_{c_\al c_\beta}^{(b_\al b_\beta)}(z) ] &= \Im\left[  G_{c_\al c_\beta}(z)  -\frac{G_{b_\al c_\beta}(z)  }{\sqrt{d-1}}-\frac{G_{c_\al b_\beta}(z)  }{\sqrt{d-1}} +\frac{G_{b_\al b_\beta}(z)  }{d-1}\right] +\OO(N^{-1/3-\fb/2}),
\end{align*}
and 
we have
\begin{align}\begin{split}\label{e:gu3}
&\phantom{{}={}}\frac{1}{d-1}\sum_{\al,\beta\in\qq{\mu}}P_{il_\al}\Im[G_{c_\al c_\beta}^{(b_\al b_\beta)}(z) ]  P_{jl_\beta}\\
&=\frac{1}{d-1}\sum_{\al,\beta\in\qq{\mu}}P_{il_\al}\Im\left[  G_{c_\al c_\beta}(z)  -\frac{G_{b_\al c_\beta}(z)  }{\sqrt{d-1}}-\frac{G_{c_\al b_\beta}(z)  }{\sqrt{d-1}} +\frac{G_{b_\al b_\beta}(z)  }{d-1}\right]   P_{jl_\beta}+\OO(N^{-1/3-\fb/4}).
\end{split}\end{align}
The claim \eqref{e:G-Y} follows from combining \eqref{e:gu2} and \eqref{e:gu3}.

\end{proof}

\begin{proof}[Proof of \Cref{l:diffG1}]

For simplicity of notations, we omit the dependence on $z$. 
 The quantity $ G_{c_\al c_\beta}^{(b_\al b_\beta)}$ can be obtained from $\wt G^{(\bT)}_{c_\al c_\beta}$ through the following steps. First, we remove $\bW=\{b_1, b_2, \cdots, b_\mu\}$, which gives $G_{c_\al c_\beta}^{(\bT \bW)}$; we then add $\bW\setminus \{b_\al, b_\beta\}$ back, which gives $G_{c_\al c_\beta}^{(\bT b_\al b_\beta)}$; finally we add $\bT$ back, which gives $G_{c_\al c_\beta}^{(b_\al b_\beta)}$. The errors from these replacements are explicit, by the Schur complement formulas \eqref{e:Schur1}:
\begin{align}
\label{e:removeW}&\wt G^{(\bT)}_{c_\al c_\beta}-G^{(\bT\bW)}_{c_\al c_\beta}=(G^{(\bT\bW)}\wt B (\wt G^{(\bT)}|_{\bW})  \wt B^\top G^{(\bT\bW)})_{c_\al c_\beta},\\
 \label{e:addW}&G^{(\bT\bW)}_{c_\al c_\beta}-G^{(\bT b_\al b_\beta)}_{c_\al c_\beta}=-(G^{(\bT\bW)}(G^{(\bT b_\al b_\beta)}|_{\bW\setminus\{b_\alpha,b_\beta\}})^{-1}   G^{(\bT b_\al b_\beta)})_{c_\al c_\beta},\\
 \label{e:schur_removeT}
&G_{c_\al c_\beta}^{(b_\al b_\beta)}-G^{(\bT b_\al b_\beta)}_{c_\al c_\beta}=(G^{(b_\al b_\beta)}(G^{(b_\al b_\beta)}|_\bT)^{-1}G^{({\yukun\bT} b_\al b_\beta)})_{c_\al c_\beta},
\end{align}
where $\wt B$ is the restriction of $\wt H$ on the vertex sets $\qq{N}\setminus (\bT\bW)\times \bW$.

Explicitly the right-hand side of \eqref{e:removeW} is given by
\begin{align*}
&\mbox{I}:=\frac{1}{d-1}\sum_{\gamma,\gamma'\in \qq{\mu}}\sum_{x\in \cN_\gamma\atop y\in \cN_{\gamma'}}G_{c_\al x}^{(\bT\bW)} \wt G^{(\bT)}_{b_\gamma b_{\gamma'}}   G^{(\bT\bW)}_{yc_\beta },
\end{align*}
where $\cN_\gamma=\{x\neq c_\gamma: x\sim b_\gamma \text{ in }\cG\}\cup \{a_\gamma\}$.

Conditioned on $\cG, \tcG\in \Omega$ and $I(\cF,\cG)=1$, \eqref{eq:local_law} gives that $|\wt G^{(\bT)}_{b_\gamma b_{\gamma'}}|\lesssim 1$ for $\gamma, \gamma'\in \qq{\mu}$. And for $x\in \cN_\gamma$ and $y\in \cN_\gamma$, we have $|G_{c_\al x}^{(\bT\bW)}|,   |G^{(\bT\bW)}_{yc_\beta }|\lesssim N^{-\fb}$. Moreover, using \eqref{e:Gest} and \eqref{e:Immbound}, we also have that $\Im[G_{c_\al x}^{(\bT\bW)}], \Im[\wt G^{(\bT)}_{b_\gamma b_{\gamma'}}],    \Im[G^{(\bT\bW)}_{yc_\beta }]\lesssim N^{-1/3+10\fo}$ hold with overwhelmingly high probability conditioned on $\cG, \tcG\in \Omega$ and $I(\cF,\cG)=1$. Thus
\begin{align}\begin{split}\label{e:I_approx}
\Im[\mbox{I}]
&\lesssim \sum_{\gamma, \gamma'\in\qq{\mu}}
\Im[G_{c_\al x}^{(\bT\bW)}] |\wt G^{(\bT)}_{b_\gamma b_{\gamma'}}   G^{(\bT\bW)}_{yc_\beta }|+|G_{c_\al x}^{(\bT\bW)}|\Im[ \wt G^{(\bT)}_{b_\gamma b_{\gamma'}} ]  |G^{(\bT\bW)}_{yc_\beta }|+|G_{c_\al x}^{(\bT\bW)} \wt G^{(\bT)}_{b_\gamma b_{\gamma'}} |\Im[  G^{(\bT\bW)}_{yc_\beta }]\\
&\lesssim \sum_{\gamma, \gamma'\in\qq{\mu}} N^{-1/3+10\fo}N^{-\fb}\lesssim N^{-1/3-\fb/2}.
\end{split}\end{align}

For the RHS of \eqref{e:addW} we denote
\begin{align}\begin{split}\label{e:defcE1beta}
\mbox{II}:
=\sum_{\gamma, \gamma'\in\qq{\mu}\setminus\{\al,\beta\}}
G^{(\bT b_\al b_\beta)}_{c_\al b_\gamma}(G^{(\bT b_\al b_\beta)}|_{\bW\setminus\{b_\alpha,b_\beta\}})^{-1}_{b_\gamma b_{\gamma'}}   G^{(\bT b_\al b_\beta)}_{b_{\gamma'} c_\beta}.
\end{split}\end{align}
Conditioned on $\cG, \tcG\in \Omega$ and $I(\cF,\cG)=1$, \eqref{eq:local_law} gives that 
for $\gamma,\gamma'\in\qq{\mu}\setminus\{\al, \beta\}$, $|G^{(\bT b_\al b_\beta)}_{c_\al b_\gamma}|, | G^{(\bT b_\al b_\beta)}_{b_{\gamma'} c_\beta}|,
|G^{(\bT b_\al b_\beta)}_{b_\gamma b_{\gamma'}}-\md(z)\delta_{\gamma\gamma'}|\lesssim N^{-\fb}$. Thus 
$|(G^{(\bT b_\al b_\beta)}|_{\bW\setminus\{b_\alpha,b_\beta\}})^{-1}_{b_\gamma b_{\gamma'}}-\delta_{\gamma\gamma'}/\md(z) |\lesssim N^{-3\fb/4} $. Moreover, using \eqref{e:Gest} and \eqref{e:Immbound}, we also have that $\Im[
G^{(\bT b_\al b_\beta)}_{c_\al b_\gamma}], \Im[G^{(\bT b_\al b_\beta)}_{b_\gamma b_{\gamma'}}],\Im[G^{(\bT b_\al b_\beta)}_{b_{\gamma'} c_\beta}]\lesssim N^{-1/3+10\fo}$ hold with overwhelmingly high probability conditioned on $\cG, \tcG\in \Omega$ and $I(\cF,\cG)=1$. It follows from \eqref{e:Ainverse} that
\begin{align*}
&\phantom{{}={}}\Im[(G^{(\bT b_\al b_\beta)}|_{\bW\setminus\{b_\alpha,b_\beta\}})^{-1}_{b_\gamma b_{\gamma'}}  ]\\
&=- \left((G^{(\bT b_\al b_\beta)}|_{\bW\setminus\{b_\alpha,b_\beta\}})^{-1}\Im[G^{(\bT b_\al b_\beta)}|_{\bW\setminus\{b_\alpha,b_\beta\}}]\overline{(G^{(\bT b_\al b_\beta)}|_{\bW\setminus\{b_\alpha,b_\beta\}})^{-1}}\right)_{b_\gamma b_{\gamma'}}\lesssim N^{-1/3+10\fo}.
\end{align*} 
We conclude that with overwhelmingly high probability conditioned on $\cG, \tcG\in \Omega$ and $I(\cF,\cG)=1$ the following holds
\begin{align}\begin{split}\label{e:II_approx}
\Im[\mbox{II}]
&=\sum_{\gamma, \gamma'\in\qq{\mu}\setminus\{\al,\beta\}}
\Im[G^{(\bT b_\al b_\beta)}_{c_\al b_\gamma}]|(G^{(\bT b_\al b_\beta)}|_{\bW\setminus\{b_\alpha,b_\beta\}})^{-1}_{b_\gamma b_{\gamma'}}   G^{(\bT b_\al b_\beta)}_{b_{\gamma'} c_\beta}|\\
&+
\sum_{\gamma, \gamma'\in\qq{\mu}\setminus\{\al,\beta\}}|G^{(\bT b_\al b_\beta)}_{c_\al b_\gamma}|\Im[(G^{(\bT b_\al b_\beta)}|_{\bW\setminus\{b_\alpha,b_\beta\}})^{-1}_{b_\gamma b_{\gamma'}}]|   G^{(\bT b_\al b_\beta)}_{b_{\gamma'} c_\beta}|\\
&+\sum_{\gamma, \gamma'\in\qq{\mu}\setminus\{\al,\beta\}}+|G^{(\bT b_\al b_\beta)}_{c_\al b_\gamma} (G^{(\bT b_\al b_\beta)}|_{\bW\setminus\{b_\alpha,b_\beta\}})^{-1}_{b_\gamma b_{\gamma'}}]|\Im[   G^{(\bT b_\al b_\beta)}_{b_{\gamma'} c_\beta}]\\
&\lesssim \sum_{\gamma, \gamma'\in\qq{\mu}\setminus\{\al,\beta\}} N^{-1/3+10\fo}N^{-\fb}\lesssim N^{-1/3-\fb/2}.
\end{split}\end{align}

Finally we denote the right-hand side of \eqref{e:schur_removeT} as 
\begin{align}\label{e:schur_removeT2}
\mbox{III}
    =-\sum_{x,y\in \bT} G^{(b_\al b_\beta)}_{c_\al x}(G^{(b_\al b_\beta)}|_\bT)^{-1}_{xy}G^{({\yukun\bT} b_\al b_\beta)}_{y c_\beta}. 
\end{align}
Conditioned on $\cG, \tcG\in \Omega$ and $I(\cF,\cG)=1$, by \eqref{eq:local_law},  for $x,y\in \bT$, $|G^{(b_\al b_\beta)}_{c_\al x}|, |G^{({\yukun\bT}b_\al b_\beta)}_{y c_\beta}|, |G^{(b_\al b_\beta)}_{xy}-(H_\bT-z -\msc\mathbb{I}^{\del})^{-1}_{xy}|\lesssim N^{-\fb}$, where $\mathbb{I}^\del_{xy}=\bm1(\dist_\cG(o,x)=\ell)\delta_{xy}$. Thus we have
\begin{align*}
|(G^{(b_\al b_\beta)}|_\bT)_{xy}^{-1}-(H^{(b_\al b_\beta)}-z -\msc\mathbb{I}^{\del})_{xy}|\lesssim N^{-3\fb/4}, \text{ for }x,y\in \bT.
\end{align*}
Moreover, using \eqref{e:Gest} and \eqref{e:Immbound}, we also have that with overwhelmingly high probability conditioned on $\cG, \tcG\in \Omega$ and $I(\cF,\cG)=1$, for $x,y\in \bT$, $\Im[G^{(b_\al b_\beta)}_{c_\al x}], \Im[G^{(b_\al b_\beta)}_{xy}], \Im[G^{({\yukun\bT} b_\al b_\beta)}_{y c_\beta}]\lesssim N^{-1/3+10\fo}$. It follows from \eqref{e:Ainverse} that
\begin{align*}
\Im[(G^{(b_\al b_\beta)}|_\bT)^{-1}_{xy} ]=- \left((G^{(b_\al b_\beta)}|_\bT)^{-1}\Im[G^{(b_\al b_\beta)}|_\bT]\overline{(G^{(b_\al b_\beta)}|_\bT)^{-1}}\right)_{xy}\lesssim N^{-1/3+10\fo}.
\end{align*} 
We conclude that with overwhelmingly high probability conditioned on $\cG, \tcG\in \Omega$ and $I(\cF,\cG)=1$,  the following holds
\begin{align}\begin{split}\label{e:III_approx}
|\Im[\mbox{III}]|
&=\sum_{x,y\in \bT} \Im[G^{(b_\al b_\beta)}_{c_\al x}]|(G^{(b_\al b_\beta)}|_\bT)^{-1}_{xy}G^{({\yukun\bT}b_\al b_\beta)}_{y c_\beta}]|+\sum_{x,y\in \bT} |G^{(b_\al b_\beta)}_{c_\al x}]\Im[(G^{(b_\al b_\beta)}|_\bT)^{-1}_{xy}]|G^{({\yukun\bT}b_\al b_\beta)}_{y c_\beta}]|\\
&+\sum_{x,y\in \bT} |G^{(b_\al b_\beta)}_{c_\al x}]|(G^{(b_\al b_\beta)}|_\bT)^{-1}_{xy}|\Im[G^{({\yukun\bT}b_\al b_\beta)}_{y c_\beta}]|\lesssim \sum_{x, y\in\bT} N^{-1/3+10\fo}N^{-\fb}\lesssim N^{-1/3+\fb/2}.
\end{split}\end{align}
\Cref{l:diffG1} follows from combining \eqref{e:removeW} -- \eqref{e:I_approx}, \eqref{e:II_approx} and \eqref{e:III_approx}.

\end{proof}

\begin{proof}[Proof of \Cref{l:ImGG}]
For simplicity of notations, we omit the dependence on $z$. 
Conditioned on $\cG\in \Omega, I(\cF,\cG)=1$, \eqref{eq:infbound0} implies
 \begin{align}\begin{split}\label{e:Gsize}
& G_{bb}, G_{cc}, G_{b'b'}, G_{c',c'}=\md(z)+\OO(N^{-\fb})=-\frac{d-1}{d-2}+\OO(N^{-\fb}), 
\\ 
&G_{bc}, G_{b'c'}=-\frac{\md(z)\msc(z)}{\sqrt{d-1}}+\OO(N^{-\fb})=-\frac{\sqrt{d-1}}{d-2}+\OO(N^{-\fb}),\; G_{bb'},G_{cc'}, G_{bc'},G_{cb'}=\OO(N^{-\fb}).
 \end{split}\end{align}
 Moreover, using \eqref{e:Gest} and \eqref{e:Immbound}, with overwhelmingly high probability conditioned on $\cG, \tcG\in \Omega$ and $I(\cF,\cG)=1$,  the following holds for any $x,y\in \qq{N}$, 
 \begin{align}\label{e:Immbound2}
 \Im[G_{xy}]\lesssim N^\fo \Im[m_N]\lesssim N^{-1/3+10\fo}. 
 \end{align}
The claim \eqref{e:Gbcc} follows from
 \begin{align*}
 &\phantom{{}={}}\Im[G_{cc}^{(b)}]=\Im\left[G_{cc}-\frac{G_{bc}G_{bc}}{G_{bb}}\right]
 =\Im[G_{cc}]-\Re\left[\frac{1}{G_{bb}}\right]\Im[G_{bc}G_{bc}]
-\Im\left[\frac{1}{G_{bb}}\right]\Re[G_{bc}G_{bc}]\\
&=\Im[G_{cc}]-\frac{2\Re[G_{bb}]}{|G_{bb}|^2}\Im[G_{bc}]\Re[G_{bc}]
+\frac{\Im[G_{bb}]}{|G_{bb}|^2}\Re[G_{bc}G_{bc}]\\
&=\Im[G_{cc}]-\frac{2\Im[G_{bc}]}{\sqrt{d-1}}+\frac{\Im[G_{bb}]}{d-1}+\OO(N^{-1/3-\fb/2}),
 \end{align*}
 where the last line follows from \eqref{e:Gsize} and \eqref{e:Immbound2}.
 
To prove \eqref{e:Gbbcc}, we start with the Schur complement formula \eqref{e:Schur1}
\begin{align}\label{e:scf}
\Im[ G_{cc'}^{(bb')}]=\Im[G_{cc'}]-\Im [(G (G|_{\{bb'\}})^{-1} G)_{cc'}],
\end{align}
where
$(G|_{\{bb'\}})^{-1}$ is given by
\begin{align}\begin{split}\label{e:rough_bb2}
    (G|_{\{bb'\}})^{-1}
    &=\frac{1}{G_{bb}G_{b'b'} -G_{bb'}^2}
    \left[
    \begin{array}{cc}
    G_{b'b'} & -G_{bb'}\\
    -G_{bb'} & G_{bb}
    \end{array}
    \right]\\
    &=\left(\sum_{j=0}^\fp \frac{G_{bb'}^{2j}}{(G_{bb}G_{b'b'})^{j+1} } +\OO\left(\frac{1}{N^2}\right)\right)
    \left[
    \begin{array}{cc}
    G_{b'b'} & -G_{bb'}\\
    -G_{bb'} & G_{bb}
    \end{array}
    \right].
\end{split}\end{align}

 By substituting \eqref{e:rough_bb2} into \eqref{e:scf}, we obtain an expansion of $\Im[G_{cc'}^{(bb')}]$. Terms in this expansion containing at least two factors of the form $\{G_{bb'},G_{cc'}, G_{bc'},G_{cb'}\}$ can be bounded by $\OO(N^\fo \Im[m_N(z)]N^{-\fb})=\OO(N^{-1/3-\fb/2})$, thanks to \eqref{e:Gsize} and \eqref{e:Immbound2}.  The remaining terms in the expansion are given by
\begin{align*}
&\phantom{{}={}} G_{cc'}-\frac{G_{cb}G_{b'b'}G_{bc'}}{G_{bb}G_{b'b'}} +\frac{G_{cb}G_{bb'}G_{b'c'}}{G_{bb}G_{b'b'}} -\frac{G_{cb'}G_{bb}G_{b'c'}}{G_{bb}G_{b'b'}} \\
&=G_{cc'}-\frac{G_{cb}G_{bc'}}{G_{bb}}-\frac{G_{cb'}G_{b'c'}}{G_{b'b'}} +\frac{G_{cb}G_{bb'}G_{b'c'}}{G_{bb}G_{b'b'}}.
\end{align*} 
This together with  \eqref{e:Gsize} and \eqref{e:Immbound2} give \eqref{e:Gbbcc}.
\end{proof}

\subsection{Proof of \Cref{t:ev_universality}}
In this section, we prove our main result, \Cref{t:ev_universality}, as a consequence of \Cref{p:G_convergence}, which establishes the weak convergence of the imaginary part of the Green’s function. Before proceeding with the proof, we state a result regarding the tightness of eigenvalues and their Stieltjes transforms, which will be proven in \Cref{s:edge_tightness}.

\begin{proposition}\label{p:eig_convergence}
The Airy$_1$ point process $\{A_1, A_2, A_3,\cdots\}$ is a simple point process. The following random variable is almost surely bounded:
\begin{align}\label{e:Airy_est}
&Y:=\sup_{x\leq 0} (1+|x|)^{-2/3}|\{i\geq 1: A_i\geq -x\}|<\infty.
\end{align}
Under the notation and assumptions of \Cref{t:ev_universality},  the sequence of random variables
\begin{align}\label{e:defYn}
&Y_N:=\sup_{x\leq 0} (1+|x|)^{-2/3}|\{i\in\qq{N}:(\cA N)^{2/3}(\lambda_i-2)\geq x\}|,
\end{align}
 is tight. Moreover,  the sequence of harmonic functions on the upper half-plane, 
 \begin{align}\label{e:sti_tight}
\{\cA^{-2/3}N^{1/3}\Im[m_N(2+w(\cA N)^{-2/3})]\}_{N\geq 1},
 \end{align}
 is also tight.
 \end{proposition}

%
%

\begin{proposition}\label{p:G_convergence}
Adopt notation and assumptions in \Cref{t:ev_universality}. The rescaled extreme eigenvalues converge to the Airy$_1$ point process:
\begin{align*}\begin{split}
&\phantom{{}={}} ((\cA N)^{2/3} ( \lambda_{2} - 2 ), (\cA N)^{2/3} ( \lambda_{3} - 2 ),\cdots, (\cA N)^{2/3} ( \lambda_{k+1} - 2 ))\Rightarrow (A_1, A_2, A_3,\cdots, A_k).
\end{split}\end{align*}
Fix any vertex $o\in \qq{N}$. The imaginary part of the rescaled Green's functions at the edge converge weakly
\begin{align}\label{e:dG_converge}
\{\cA^{-2/3}N^{1/3}(\Im[G_{ij}(2+w(\cA N)^{-2/3})]\}_{i,j\in \bS}\Rightarrow \{\Im[\bfG_{ij}(w)]\}_{i,j\in \bS},
\end{align}
where $\bS$ is the vertex set of $\cB_r(o,\cG)=\cB_r(o,\cX_d)$, 
\begin{align*}
\Im[\bfG_{ij}(w)]=\sum_{s\geq 1}\Im\left[\frac{Z_s(i)Z_s(j)}{A_s -w}\right], \quad i,j\in \bS,
\end{align*}
and $\{Z_s\}_{s\geq 1}$ are independent copies of the Gaussian wave $\Psi$ \eqref{e:edge_Gaussian_wave}.

\end{proposition}

\begin{proof}[Proof of \Cref{t:ev_universality}]
Using the Skorokhod representation theorem, we can take a subsequence, $N_1<N_2<N_3\cdots$, such that along this subsequence, almost surely, $Y_N$ converges to a random variable $Y$, $\lambda_1, \lambda_2,\cdots$ converges to the Airy$_1$ point process, and 
$\{\cA^{-2/3}N^{1/3}(\Im[G_{ij}(2+w(\cA N)^{-2/3})]\}_{i,j\in \bS, w\in \bC^+}$ converges (uniformly on compact sets) to the  random process $\{\Im[\bfG_{ij}(w)]\}_{i,j\in \bS, w\in \bC^+}$.

For any $i,j\in \bS$, near the spectral edge $2$,  the imaginary part of the rescaled Green's functions $\cA^{-2/3}N^{1/3}(\Im[G_{ij}(2+w(\cA N)^{-2/3})]$ is explicitly given by 
\begin{align*}
\cA^{-2/3}N^{1/3}(\Im[G_{ij}(2+w(\cA N)^{-2/3})]
&=\Im \sum_{s=1}^N \frac{N\bmu_s(i)\bmu_s(j)}{ (\cA N)^{2/3}(\lambda_s-2)- w}\\
&=\sum_{s=1}^N \frac{\Im[w]N\bmu_s(i)\bmu_s(j)}{ ((\cA N)^{2/3}(\lambda_s-2)- \Re[w])^2+\Im[w]^2}.
\end{align*}
It is the Possion kernel integral with the measure 
\begin{align*}
\mu_N:=\sum_{s=1}^N N\bmu_s(i)\bmu_s(j) \delta_{(\cA N)^{2/3}(\lambda_s-2)},
\end{align*}
and harmonic for $w\in \bC^+$. And similarly $\Im \bfG_{ij}(w)$
\begin{align}\label{e:Gij}
\Im[\bfG_{ij}(w)]=\sum_{s\geq 1} \frac{\Im[w]Z_s(i)Z_s(j)}{ (A_s- \Re[w])^2+\Im[w]^2}, \quad i,j\in \bS,
\end{align}
is the Possion kernel integral with the measure 
\begin{align*}
\mu:=\sum_{s\geq 1} Z_s(i)Z_s(j) \delta_{A_s}.
\end{align*}
Since $Z_s$ are independent copies of Gaussian waves $\Psi$, $Z_s(i)$ are standard Gaussian random variables (for different $i$, they are correlated). Together with \eqref{e:Airy_est}, we conclude the RHS of \eqref{e:Gij} is almost surely absolutely convergent. 

Next we check the assumptions in \Cref{l:boundary_measure} along the sequence $N_1<N_2<N_3\cdots$. Along this sequence, we can assume that there exists a large $M>0$ such that $(\cA N)^{2/3}(\lambda_s-2)\leq M$ for each $1\leq s\leq N$. Then by Cauchy-Schwartz inequality
\begin{align*}
 \int_\bR \frac{\rd|\mu_N|(x)}{1+x^2}
 &=\sum_{s=1}^N \frac{N|\bmu_s(i)\bmu_s(j)| }{1+((\cA N)^{2/3}(\lambda_s-2))^2}
 \leq \frac{1}{2}\sum_{s=1}^N \frac{N(\bmu^2_s(i)+\bmu^2_s(j)) }{1+((\cA N)^{2/3}(\lambda_s-2))^2}\\
 &\leq N^{1/3}\Im G_{ii}(2+(2M+\ri)N^{-2/3})+N^{1/3}\Im G_{jj}(2+(2M+\ri)N^{-2/3}),
\end{align*}
which are uniformly bounded. By the same argument, $ \int_\bR {\rd|\mu|(x)}/(1+x^2)$ is also bounded. We conclude from \Cref{l:boundary_measure}, along the sequence $N_1<N_2< N_3,\cdots$, $\mu_N$ converges to $\mu$ in vague topology almost surely. In particular, $(\cA N)^{2/3} ( \lambda_{s+1} - 2 )$ converges to $A_s$ almost surely, and $N\bmu_{s+1}(i)\bmu_{s+1}(j)$ converges to $Z_s(i)Z_s(j)$ almost surely for each $s\geq 1$. This finishes the proof of \Cref{t:ev_universality}.
\end{proof}

\begin{proof}[Proof of \Cref{p:G_convergence}]
Since the switched graph $\wt \cG$ and the original graph $\cG$ have the same law, we prove the statements of \Cref{p:G_convergence} for the switched graph $\wt \cG$. We denote its eigenvalues as $\wt\la_1=d/\sqrt{d-1}\geq \wt \la_2\geq \cdots\geq \wt \la_N$, its Green's function as $\wt G(z)$, and the Stieltjes transform of its empirical eigenvalue distribution as $\wt m_N(z)$. We divide the proof into the following five steps. 

\textit{Step 1.}\,By \Cref{t:universality} and the tightness of the normalized Stieltjes transform $\{\cA^{-2/3}N^{1/3}(\Im[m_N(2+w(\cA N)^{-2/3})]\}_{N\geq 1}$ near the spectral edge \eqref{e:sti_tight}, we have the joint weak convergence
\begin{align}\begin{split}\label{e:weak1}
&((\cA N)^{2/3} ( \lambda_{2} - 2 ), (\cA N)^{2/3} ( \lambda_{3} - 2 ), (\cA N)^{2/3} ( \lambda_{4} - 2 ), \cdots)\Rightarrow (A_1, A_2, A_3,\cdots),\\
&\cA^{-2/3}N^{1/3}\Im[m_N(2+w(\cA N)^{-2/3})]\Rightarrow \Im[\bfm(w)]:=\sum_{s\geq 1}\frac{\Im[w]}{|A_s-w|^2},
\end{split}\end{align}
where $A_1, A_2, \cdots$ are the points of the Airy$_1$ point process.
For the Stieltjes transform $\wt m_N(z)$ of the switched graph, with high probability we have the following bound 
\begin{align*}
&\phantom{{}={}}\cA^{-2/3}N^{1/3}|m_N(2+w(\cA N)^{-2/3})-\wt m_N(2+w(\cA N)^{-2/3})|\\
&\lesssim (d-1)^\ell N^{1/3+\fo}  \frac{\Im[m_N(2+w(\cA N)^{-2/3})]}{N^{1/3}\Im[w]}
\lesssim \frac{(d-1)^\ell N^\fo}{\Im[w]} N^{-1/3+8\fo}\lesssim N^{-\fb},
\end{align*} 
where in the first statement we used \eqref{e:tmmdiff}; in the second statement we used \eqref{e:Immbound}. The above estimate, together with the  weak convergence of the Stieltjes transform $\wt m_N(z)$ in \eqref{e:weak1}, imply
\begin{align}\begin{split}\label{e:weak2}
&\cA^{-2/3}N^{1/3}\Im[\wt m_N(2+w(\cA N)^{-2/3})]\Rightarrow \Im[\bfm(w)].
\end{split}\end{align}
Then \Cref{l:boundary_measure} and \eqref{e:weak2} imply the weak convergence of the eigenvalues of the switched graph $\wt \cG$,
\begin{align*}
&((\cA N)^{2/3} ( \wt \lambda_{2} - 2 ), (\cA N)^{2/3} ( \wt \lambda_{3} - 2 ), (\cA N)^{2/3} ( \wt \lambda_{4} - 2 ), \cdots)\Rightarrow (A_1, A_2, A_3,\cdots).
\end{align*}

\textit{Step 2.} Next we make some preparations for the proof of \eqref{e:dG_converge}, the weak convergence of the imaginary part of the Green's functions at the edge for $\wt G$. Thanks to \eqref{e:theset}, we can condition on the event
\begin{align}\label{e:theset2}
\cG, \tcG\in \Omega, \quad I(\cF,\cG)=1.
\end{align}
For any $i\in \bS$, we introduce the  vectors $\bmv^{(i)}\in \bR^N$ such that
\begin{align*}
\bmv^{(i)}(c_\al)=\frac{P_{i l_\al}}{\sqrt{d-1}}, \quad \bmv^{(i)}(b_\al)=-\frac{P_{i l_\al}}{d-1}, \text{ for } \al\in\qq{\mu} \quad \bmv^{(i)}(k)=0, \text{ for } k\in\qq{N}\setminus\{b_\al, c_\al\}_{\al\in\qq{\mu}}.
\end{align*}
We recall from \Cref{s:local_resampling} that the edges $\{b_\al, c_\al\}$ for $\al\in\qq{\mu}$ are randomly chosen from the graph $\cG^{(\bT)}$. Moreover for $i\in \bS$, \eqref{e:Pexp} gives 
\begin{align}\label{e:Pbound2}
\max_{\al\in \qq{\mu}}|P_{i l_\al}|\lesssim (d-1)^{-\ell/2}.
\end{align}
With this notation we can rewrite the RHS of \eqref{e:G-Y} as 
\begin{align}\begin{split}\label{e:rewrite}
&\phantom{{}={}}\frac{1}{d-1}\sum_{\al,\beta\in\qq{\mu}}P_{il_\al}\left(  G_{c_\al c_\beta}(z)-\frac{G_{b_\al c_\beta}(z)}{\sqrt{d-1}}-\frac{G_{c_\al b_\beta}(z)}{\sqrt{d-1}} +\frac{G_{b_\al b_\beta}(z)}{d-1}\right)   P_{jl_\beta}
\\
&=\langle \bmv^{(i)} , G(z)\bmv^{(j)}\rangle=\frac{1}{N}\sum_{s=1}^N \frac{N\langle \bmv^{(i)}, \bmu_s\rangle \langle \bmv^{(j)}, \bmu_s\rangle}{\lambda_s-z}.
\end{split}\end{align}

\textit{Step 3.} In this step, we shall show that condition on $\cG\in \Omega$, $\{\sqrt N \langle \bmv^{(i)}, \bmu_s\rangle \}_{i\in \bS, 2\leq s\leq \sqrt{ N}}$ are asymptotically joint Gaussian, with respect to the randomness of the switching data $\bfS=\{(l_\al,a_\al), (b_\al, c_\al)\}_{\al\in \qq{\mu}}$. Explicitly,
\begin{align*}
\sqrt{N}\langle \bmv^{(i)}, \bmu_s \rangle
=\sum_{\al\in\qq{\mu}}\frac{P_{il_\al}}{\sqrt{d-1}}X_s(\al), \quad\mbox{and}\quad  X_s(\al):=\sqrt N\left(\bmu_s (c_\al)-\frac{\bmu_s (b_\al)}{\sqrt{d-1}}\right).
\end{align*}

Given $ 2\leq s\leq \sqrt{N}$, $X_s(\al)$ for $\al\in\qq{\mu}$ are identically independent random variables, with 
\begin{align}\begin{split}\label{e:1moment}
&\phantom{{}={}}\bE_\bfS[X_s(\al)]=\bE_\bfS\left[\sqrt N\left(\bmu_s (c_\al)-\frac{\bmu_s (b_\al)}{\sqrt{d-1}}\right)\right]
=\frac{\sqrt{N}}{Nd-\OO((d-1)^\ell)}\sum_{b\sim c \text{ in } \cG^{(\bT)}}\bmu_s (c)-\frac{\bmu_s (b)}{\sqrt{d-1}}\\
&=\frac{\sqrt{N}}{Nd-\OO((d-1)^\ell)}\sum_{b\in\qq{N}}d\bmu_s (b)-\frac{d\bmu_s (b)}{\sqrt{d-1}} +\OO\left(\frac{(d-1)^\ell  N^\fo}{Nd}\right)=\OO\left(\frac{(d-1)^\ell  N^\fo}{Nd}\right).
\end{split}\end{align} 
Here in the second step we used that $(b_\al, c_\al)$ are uniformly distributed over $\cG^{(\bT)}$; in the last line we used $\langle \bmu_s, \bm1\rangle=0$, and conditioned on $\cG\in \Omega$, delocalization of eigenvectors $\|\bmu_s\|_\infty\lesssim N^{\fo}/\sqrt N$ from \eqref{e:delocalization} holds.

Now let us look at the covariance. If $\al\neq \beta\in\qq{\mu}$, $X_s(\al)$ and $X_t(\beta)$ are independent, and \eqref{e:1moment} implies
\begin{align}\label{e:2moment}
\bE_\bfS[X_s(\al) X_t(\beta)]=\OO\left(\frac{(d-1)^{2\ell} N^{2\fo}}{N^2}\right).\end{align}
If $\al=\beta$, we have 
\begin{align}
&\phantom{{}={}}\bE_\bfS[X_s(\al) X_t(\al)]
=N\bE_\bfS\left[\left(\bmu_s (c_\al)-\frac{\bmu_s (b_\al)}{\sqrt{d-1}}
\right)\left(\bmu_t (c_\al)-\frac{\bmu_t (b_\al)}{\sqrt{d-1}}
\right) \right]\label{e:3moment}\\
&=N\bE_\bfS\left[\bmu_s (c_\al)\bmu_t (c_\al)-\frac{\bmu_s (b_\al)\bmu_t (c_\al)+\bmu_s (c_\al)\bmu_t (b_\al)}{\sqrt{d-1}}
+ \frac{\bmu_s (b_\al)\bmu_t (b_\al)}{d-1}\right]\nonumber\\
&=\frac{N}{Nd-\OO((d-1)^\ell)}\sum_{b\sim c\in\qq{N}}\bmu_s (c)\bmu_t (c)-\frac{\bmu_s (b)\bmu_t (c)+\bmu_s (c)\bmu_t (b)}{\sqrt{d-1}}
+ \frac{\bmu_s (b)\bmu_t (b)}{d-1} +\OO\left(\frac{(d-1)^\ell  N^{2\fo}}{Nd}\right)\nonumber\\
&=\frac{N}{Nd-\OO((d-1)^\ell)}\sum_{b\in\qq{N}}d\bmu_s (b)\bmu_t (b)-2\lambda_s \bmu_s (b)\bmu_t (b)
+ \frac{d\bmu_s (b)\bmu_t (b)}{d-1} +\OO\left(\frac{(d-1)^\ell  N^{2\fo}}{Nd}\right)\nonumber\\
&=\frac{d^2-2(d-1)\lambda_s}{d(d-1)}\delta_{st}+\OO\left(\frac{(d-1)^\ell  N^{2\fo}}{Nd}\right)=\frac{\delta_{st}}{\cA}+\OO\left(|\lambda_s-2|+\frac{(d-1)^\ell  N^{2\fo}}{Nd}\right)=\frac{\delta_{st}}{\cA}+\OO\left(\frac{1}{N^{1/3}}\right)\nonumber\,.
\end{align}
Here in the third step we used that $(b_\al, c_\al)$ are uniformly distributed over $\cG^{(\bT)}$; in the fourth and fifth steps we used that
\begin{align*}
\sum_{b: b\sim c}\bmu_s(b)=\lambda_s\sqrt{d-1}\bmu_s(c), \quad \sum_{b\in\qq{N}} \bmu_s(b)\bmu_t(b)=\delta_{st};
\end{align*}
in the last two steps we used $\cA=d(d-1)/(d-2)^2$ and the optimal rigidity \eqref{e:optimal_rigidity} that for $s\leq \sqrt{N}$, $|\lambda_s-2|\lesssim (\sqrt{N}/N)^{2/3}\lesssim N^{-1/3}$.

Conditioned on $\cG\in \Omega$, by the delocalization of eigenvectors \eqref{e:delocalization}, we have $|X_\al|\lesssim N^{\fo}$. Thus for any fixed $k\geq 1$, indices $i_1, i_2,\cdots, i_k\in \bS$,  $2\leq s_1, s_2, \cdots, s_k\leq \sqrt N$, and numbers $|t_1|, |t_2|, \cdots, |t_k|\lesssim (d-1)^{\ell/4}$, it holds
\begin{align}\label{e:bb1}
 \left|\sum_{a=1}^k\frac{t_a P_{i_a l_\al }}{\sqrt{d-1}} X_{s_a}(\al)\right|
 \lesssim \max_{a}|t_a| \times \frac{k N^\fo}{(d-1)^{\ell/2}}\lesssim N^{-\fo}.
 \end{align}
 The moment generating function of $\{\sqrt N \langle \bmu_{s_a}, \bmv^{(i_a)}\rangle\}_{1\leq a\leq k}$ satisfies
\begin{align}\begin{split}\label{e:moment_gene0}
&\phantom{{}={}}\bE_\bfS\left[\exp\left\{\sum_{a=1}^{k} t_a\sqrt N \langle\bmu_{s_a}, \bmv^{(i_a)}\rangle\right\}\right]=\bE_\bfS\left[\exp\left\{\sum_{a=1}^{k} t_a \sum_{\al \in\qq{\mu}} \frac{P_{i_a l_\al }}{\sqrt{d-1}} X_{s_a}(\al)\right\}\right]
\\
&=\bE_\bfS\left[\prod_{\al\in \qq{\mu}}\exp\left\{\sum_{a=1}^{k}  \frac{t_a P_{i_a l_\al }}{\sqrt{d-1}} X_{s_a}(\al)\right\}\right]\\
&=\prod_{\al\in \qq{\mu}}\bE_\bfS\left[1+ \sum_{a=1}^{k}  \frac{t_a P_{i_a l_\al }}{\sqrt{d-1}} X_{s_a}(\al)+\frac{1}{2}\left(\sum_{a=1}^{k}  \frac{t_a P_{i_a l_\al }}{\sqrt{d-1}} X_{s_a}(\al)\right)^2 +\OO\left(\left|\sum_{a=1}^{k}  \frac{t_a P_{i_a l_\al }}{\sqrt{d-1}} X_{s_a}(\al)\right|^3\right)\right],
\end{split}\end{align}
where $\bE_\bfS[\cdot]$ is the expectation with respect to the randomness of switching data $\bfS$ as in \Cref{def:PS}.

By \eqref{e:1moment} and \eqref{e:Pbound2}, we have
\begin{align}\label{e:bb2}
\bE_\bfS\left[\sum_{a=1}^{k}  \frac{t_a P_{i_a l_\al }}{\sqrt{d-1}} X_{s_a}(\al)\right]\lesssim \max_a |t_a| \times (d-1)^{-\ell/2} \frac{k (d-1)^\ell  N^\fo}{Nd}\lesssim \max_a |t_a| \times  \frac{k (d-1)^{\ell/2}  N^\fo}{N}.
\end{align}
By \eqref{e:2moment}, \eqref{e:3moment} and \eqref{e:Pbound2}, we have
\begin{align}\label{e:bb3}
\bE_\bfS\left[\frac{1}{2}\left(\sum_{a=1}^{k}  \frac{t_a P_{i_a l_\al }}{\sqrt{d-1}} X_{s_a}(\al)\right)^2 \right]
=\frac{1}{2}\sum_{a,b=1}^{k}  \frac{t_a t_b P_{i_a l_\al} P_{i_a l_\beta}\delta_{s_a s_b}}{\cA(d-1)}+\OO\left( \max_a |t_a|^2 \times \frac{k^2(d-1)^\ell}{N^{1/3}}\right).
\end{align}

Then we can plug \eqref{e:bb1}, \eqref{e:bb2} and \eqref{e:bb3} into \eqref{e:moment_gene0}
\begin{align}\begin{split}\label{e:moment_gene2}
&\phantom{{}={}}\bE_\bfS\left[\exp\left\{\sum_{a=1}^{k} t_a\sqrt N \langle\bmu_{s_a}, \bmv^{(i_a)}\rangle\right\}\right]
\\
&=\prod_{\al\in \qq{\mu}}\bE_\bfS\left[1+\frac{1}{2}\sum_{a,b=1}^{k}  \frac{t_a t_b P_{i_a l_\al} P_{i_a l_\beta}\delta_{s_a s_b}}{\cA(d-1)} +\OO\left(\max_{a}(1+|t_a|)^3 \times \frac{k^3 N^{3\fo}}{(d-1)^{3\ell/2}}\right)\right]\\
&=\exp\left\{
\frac{1}{2}\sum_{a,b=1}^{k} \sum_{\al\in \qq{\mu}}  \frac{t_a t_b P_{i_a l_\al} P_{i_a l_\beta}\delta_{s_a s_b}}{\cA(d-1)}  +\OO\left(\max_{a}(1+|t_a|)^3 \times \frac{k^3 N^{3\fo}}{(d-1)^{\ell/2}}\right)
\right\}.
\end{split}\end{align}
Therefore $\{\sqrt N \langle\bmu_{s_a}, \bmv^{(i_a)}\rangle\}_{1\leq a\leq k}$ converges to joint real gaussian random variables.

In the special case, for $k=1$, \eqref{e:moment_gene2} reduces to the following 
\begin{align*}
\bE_\bfS\left[e^{t \sqrt N\langle\bmu_s, \bmv^{(i)}\rangle}\right]=\exp\left\{
\frac{1}{2} \sum_{\al\in \qq{\mu}}  \frac{t^2 P^2_{i l_\al}}{\cA(d-1)}  +\OO\left( \frac{(1+|t|)^3 N^{3\fo}}{(d-1)^{\ell/2}}\right)
\right\}=e^{\OO(t^2)},
\end{align*}
provided $|t|\lesssim (d-1)^{\ell/10}$.
Combining with a Markov's inequality, we conclude
\begin{align}\label{e:vconcentration}
\bP_\bfS(|\sqrt N\langle\bmu_s, \bmv^{(i)}\rangle|\geq t)\leq \fC e^{-t^2/\fC}, 
\end{align}
for $t\leq (d-1)^{\ell/10}$ and $N$ large enough.

\textit{Step 4.} Next we identify the covariance structure with that of the Gaussian wave $\Psi$ \eqref{e:edge_Gaussian_wave}.
For any $i,j\in \bS$, \eqref{e:3moment} gives
\begin{align}\begin{split}\label{e:to_GW}
&\phantom{{}={}}\bE_\bfS[N\langle \bmv^{(i)}, \bmu_s \rangle \langle \bmv^{(j)}, \bmu_s\rangle]
= \sum_{\al\in \qq{\mu}} \frac{P_{il_\al}P_{j l_\al}}{(d-1)\cA}+\OO\left(\frac{(d-1)^\ell}{N^{1/3}}\right)\\
&= \frac{1}{(d-1)} \lim_{\varepsilon\rightarrow 0}\sum_{\al\in \qq{\mu}}  \frac{P_{il_\al}\Im[\msc(2+\ri \varepsilon)] P_{j l_\al}}{\Im[\md(2+\ri \varepsilon)]}+\OO\left(\frac{(d-1)^\ell}{N^{1/3}}\right)
\\
&=\lim_{\varepsilon\rightarrow 0}\frac{-(P\Im[H|_\bT-2-\msc(2+\ri \varepsilon)\bI^\partial]P)_{ij}}{\Im[\md(2+\ri \varepsilon)]}+\OO\left(\frac{(d-1)^\ell}{N^{1/3}}\right)\\
&=\lim_{\varepsilon\rightarrow 0}\frac{\Im[(H|_\bT-(2+\ri\varepsilon) -\msc(2+\ri \varepsilon))^{-1}_{ij}] +\OO(\varepsilon)}{\Im[\md(2+\ri \varepsilon)]}+\OO\left(\frac{(d-1)^\ell}{N^{1/3}}\right)\\
&=\lim_{\varepsilon\rightarrow 0}\frac{1}{\Im[\md(2+\ri \varepsilon)]}\Im\left[\md(2+\ri \varepsilon) \left(-\frac{\msc(2+\ri\varepsilon)}{\sqrt{d-1}}\right)^r\right]+\OO\left(\frac{(d-1)^\ell}{N^{1/3}}\right)\\
&=\frac{1}{(d-1)^{r/2}}\left(1+ \frac{(d-2)r}{d}\right)+\OO\left(\frac{(d-1)^\ell}{N^{1/3}}\right) ={\rm Cov}[\Psi(i)\Psi(j)]+\OO\left(\frac{(d-1)^\ell}{N^{1/3}}\right),
\end{split}\end{align}
where in the second statement we used \eqref{e:mscmd} that $\Im[\md(2+\ri\varepsilon)/\msc(2+\ri \varepsilon)]\rightarrow \cA$, as $\varepsilon\rightarrow 0$; the third statement follows from the definition $\bI^\del_{xy}=\delta_{xy}\bm1(\dist_{\cT}(o,x)=\ell)$; the fourth statement follows from replacing $P$ by $P(\cT, 2+\ri \varepsilon, \msc(2+\ri\varepsilon))$, and using \eqref{e:PPdiff}; the fifth statement follows from \eqref{e:Pexp}; the last two statements follow from \eqref{e:edge_Gaussian_wave}.

\textit{Step 5.} In this step we wrap up the proof of  \eqref{e:dG_converge}. Let $z=2+(\cA N)^{-2/3}w$, and large $\fj\geq 1$, which will be chosen later. Conditioned on $\cG, \tcG\in \Omega$ and $ I(\cF,\cG)=1$, by \eqref{e:G-Y} and \eqref{e:rewrite}, with overwhelmingly high probability, we have the following decomposition for the Green's function near the spectral edge
\begin{align}\begin{split}\label{e:G_decompose}
&\phantom{{}={}}\cA^{-2/3}N^{1/3}\Im[\wt  G_{ij}(z)]=\sum_{s=1}^N\frac{\Im[w] N \langle \bmu_s, \bmv^{(i)} \rangle \langle \bmu_s, \bmv^{(j)} \rangle }{|(\cA N)^{2/3}(\lambda_s-2)-w|^2}+\OO(N^{-\fb/4})\\
&=\sum_{1\leq s\leq \fj}\frac{\Im[w] N \langle \bmu_s, \bmv^{(i)} \rangle \langle \bmu_s, \bmv^{(j)} \rangle }{|(\cA N)^{2/3}(\lambda_s-2)-w|^2}
+\sum_{\fj<s\leq N}\frac{\Im[w] N \langle \bmu_s, \bmv^{(i)} \rangle \langle \bmu_s, \bmv^{(j)} \rangle }{|(\cA N)^{2/3}(\lambda_s-2)-w|^2}+\OO(N^{-\fb/4}).
\end{split}\end{align}
To bound the second term in the last line of \eqref{e:G_decompose}, we recall from \eqref{e:defYn} that 
\begin{align}\label{e:YNuse}
\#\{s: (\cA N)^{2/3}(\lambda_s-2)\geq -x\}\leq (1+|x|)^{3/2}Y_N.
\end{align}
By  \Cref{p:eig_convergence}, the sequence $\{Y_N\}_{N\geq 1}$ is tight. For any $\varepsilon>0$, we can find a large $\fC>0$, such that $\bP(Y_N\leq \fC)\geq 1-\varepsilon$ for any $N\geq 1$. Conditioned on $Y_N\leq \fC$, \eqref{e:YNuse} implies
\begin{align*}
(\cA N)^{2/3}(\lambda_s-2)\leq 1-\left(\frac{s}{\fC}\right)^{2/3}.
\end{align*}
provided $s\geq \fC$. For $\fj\geq \fC(1+|w|)^{3/2}$ large enough, and $s>\fj$ we conclude that
\begin{align}\label{e:eig_bbha}
|(\cA N)^{2/3}(\lambda_s-2)-w|^2\geq \frac{1}{4}\left(\frac{s}{\fC}\right)^{4/3}.
\end{align}
Moreover, for $\fj\leq s\leq (d-1)^\ell$, \eqref{e:vconcentration} implies
\begin{align}\label{e:ev_concentration}
\bP(\sqrt N|\langle \bmu_s, \bmv^{(i)}\rangle|)\leq s^{1/12}) \geq 1-s^{-(D+1)}-N^{-\fC},
\end{align}
provided $\fj$ is large enough. 
For $s\geq (d-1)^\ell$ \eqref{e:ev_concentration} follows trivially by the delocalization of eigenvectors \eqref{e:delocalization}. 
By an union bound, we have that with probability $1-\OO(\fj^{-D}+N^{-\fC})$,
\begin{align}\label{e:dv_bbha}
\sqrt N|\langle \bmu_s, \bmv^{(i)}\rangle|)\leq s^{1/12}.
\end{align}
for all $\fj<s\leq N$. By plugging \eqref{e:eig_bbha} and \eqref{e:dv_bbha} into the second term in \eqref{e:G_decompose}, we conclude that 
\begin{align*}
\left|\sum_{\fj<s\leq N}\frac{\Im[w] N \langle \bmu_s, \bmv^{(i)} \rangle \langle \bmu_s, \bmv^{(j)} \rangle }{|(\cA N)^{2/3}(\lambda_s-2)-w|^2}\right|
\leq 
\sum_{\fj<s\leq N}
\frac{\Im[w] 4\fC^{4/3} s^{1/6}}{ s^{4/3}}\lesssim \frac{4\fC^{4/3}\Im[w]}{\fj^{1/6}},
\end{align*}
holds with probability at least $1-\OO(\fj^{-D}+\varepsilon+N^{-\fC})$. Thus
\begin{align}\begin{split}
&\phantom{{}={}}\cA^{-2/3}N^{1/3}\Im[\wt G_{ij}(z)]=\sum_{1\leq s\leq \fj}\frac{\Im[w] N \langle \bmu_s, \bmv^{(i)} \rangle \langle \bmu_s, \bmv^{(j)} \rangle }{|(\cA N)^{2/3}(\lambda_s-2)-w|^2}+\OO\left(\frac{\fC^{4/3}\Im[w]}{\fj^{1/6}}\right).
\end{split}\label{e:Glimit}
\end{align}
holds with probability at least $1-\OO(\fj^{-D}+\varepsilon+N^{-\fC})$. By sending $N\rightarrow \infty$, using \eqref{e:weak1} and \eqref{e:to_GW}, the first term on the RHS of \eqref{e:Glimit} converges weakly to 
\begin{align} \label{3.56}
\sum_{1\leq s\leq \fj}\frac{\Im[w] Z_s{(i)}Z_s{(j)} }{|A_s-w|^2},
\end{align}
where $\{Z_s\}_{s\geq 1}$ are independent copies of the Gaussian wave $\Psi$ from \eqref{e:edge_Gaussian_wave}. By the same argument as in \eqref{e:Glimit}, with probability at least $1-\OO(\fj^{-D}+\varepsilon)$, we have
\begin{align}\label{e:Glimit2}
\sum_{1\leq s\leq \fj}\frac{\Im[w] Z_s^{(i)}Z_s^{(j)} }{|A_s-w|^2}
=\sum_{s\geq 1}\frac{\Im[w] Z_s^{(i)}Z_s^{(j)} }{|A_s-w|^2} +\OO\left(\frac{\fC^{4/3}\Im[w]}{\fj^{1/6}}\right)=\Im[{\mathbf G}_{ij}(w)]+\OO\left(\frac{\fC^{4/3}\Im[w]}{\fj^{1/6}}\right).
\end{align}
Combining \eqref{e:Glimit}--\eqref{e:Glimit2} yields the weak convergence of the Green's function
\begin{align*}
\cA^{-2/3}N^{1/3}\Im[\wt G_{ij}(2+w(\cA N)^{-2/3})]\Rightarrow \Im[{\mathbf G}_{ij}(w)].
\end{align*}
This finishes the proof.
\end{proof}

\section{Edge Eigenvalues and Stieltjes Transform}
\label{s:edge_tightness}

In this section we prove \Cref{p:eig_convergence}. The proof follows from two key propositions. The first one provides concentration results for the Stieltjes transform down to the optimal scale, while the second proposition establishes concentration estimates for the extreme eigenvalues.

\begin{proposition}\label{p:sti_tightness}
There exists a sufficiently large constant $\fC>0$, such that for any $\fC N^{-2/3}\leq \gamma\leq (M/N)^{2/3} \leq N^{-\fg}$ and for any large $D>0$, the following holds.
Conditioned on $\cG\in \Omega$, with probability at least $1-\OO(M^{-D})$, for sufficiently large $N$, we have
\begin{align}\label{e:mdiff_tight}
|m_N(z)-\md(z)|\geq \frac{M}{N\Im[z]}\quad \mbox{ for}\quad  |\Re[z]-2|\leq 2\gamma, \; \gamma\leq \Im[z]\leq 2(M/N)^{2/3},
\end{align}
and
\begin{align}\label{e:mdiff_tight2}
|\Im[m_N(z)-\md(z)]|\geq \frac{M}{N\Im[z]} \quad \mbox{ for}\quad |\Re[z]-2|\leq 2\gamma, \; \Im[z]\leq 2(M/N)^{2/3}. 
\end{align}
\end{proposition}

\begin{proposition}\label{p:eig_concentration}
There exists a large $\fC>1$ such that for any $x\geq 1$ the following holds. Conditioned on $\cG\in \Omega$, with probability at least $1-\OO(x^{-D})$, for sufficiently large $N$, we have:
the second largest eigenvalue satisfies
\begin{align}\label{e:lambda2bound}
N^{2/3}(\lambda_2-2)\leq x,
\end{align}
and the number of eigenvalues on the interval $[2-N^{-2/3}x, 2+N^{-2/3}x]$ satisfies
\begin{align}\label{e:lambdajbound}
\#\{s: N^{2/3}|\lambda_s-2|\leq  x\} \leq \fC x^{3/2}.
\end{align}
\end{proposition}

\begin{proof}[Proof of \Cref{p:eig_convergence}]
The two statements \eqref{e:lambda2bound} and \eqref{e:lambdajbound} together implies that  for any $x\leq -1$ and $t\geq \cA \fC$ large enough
\begin{align*}
&\phantom{{}={}}\bP(|x|^{-3/2}\#\{s: ( \cA N)^{2/3}(\lambda_s-2)\geq x\} \geq t, \cG\in \Omega)\leq \bP(N^{2/3}(\lambda_2-2)\geq (t/\fC)^{2/3} |x|, \cG\in \Omega )\\
&+\bP(\#\{s:  N^{2/3}|\lambda_s-2|\leq (t/\fC )^{2/3} |x|\} \geq t|x|^{3/2} )\lesssim \frac{1}{(t/\fC)^{2D/3}|x|^D}.
\end{align*}
As a consequence
\begin{align*}
\bP(Y_N\geq 2t)
&=\bP(\exists x\leq 0, |\{s: (\cA N)^{2/3}(\lambda_s-2)\geq x\}|\geq 2t(1+|x|)^{2/3})\\
&\leq \sum_{k=-\infty}^{-1}\bP(|k|^{-3/2}\#\{s: (\cA N)^{2/3}(\lambda_s-2)\geq k\} \geq t)\lesssim \sum_{k=-\infty}^{-1}\frac{1}{(t/\fC)^{2D/3}|k|^D}\lesssim \frac{1}{(t/\fC)^{2D/3}}.
\end{align*}
We conclude that $Y_N$ is tight.

For \eqref{e:Airy_est}, since the extreme eigenvalues of random 
$d$-regular graphs converge to the Airy$_1$ point process, the claim can be established in two ways: either by taking the limit of the probability bounds in \eqref{e:lambda2bound} and \eqref{e:lambdajbound}, or through a much stronger estimate provided in \cite[Proposition 2.4]{zhong2024large}. We omit the proof. 

Finally the tightness of the imaginary part of the Stieltjes transform follows from \eqref{e:mdiff_tight2}.
\end{proof}

\subsection{Tightness of Stieltjes Transform}
\label{s:tightness}

In this section, we establish moment estimates for $|m_N(z)-\md(z)|$ down to the optimal scale. These estimates are crucial for proving \Cref{p:sti_tightness} and \Cref{p:eig_concentration}. Similar results at this scale have been obtained for the 
$\beta$-ensembles in \cite{bourgade2022optimal} using loop equations, and for the infinite particle Dyson's Brownian motion in \cite{huang2024convergence} using stochastic calculus. Following the arguments in  \cite{bourgade2022optimal}, our proof of \Cref{p:m-mdmoment} relies on the estimate \eqref{e:QYQbound2}, which is itself a consequence of loop equations.

\begin{proposition}\label{p:m-mdmoment}
There exists a large constant $\fC$ such that for any 
$|z-2|\leq N^{-\fg}$ with $\Im[z]=\eta$, $\kappa\deq|\Re[z]-2|\wedge |\Re[z]+2|$ and $N\eta\sqrt{\kappa+\eta}\geq \fC$, the following moment estimates hold.
\begin{enumerate}
\item If $\Re[z]\in [-2-\eta, 2+\eta]$
\begin{align}\label{e:moment1}
\bE[\bm1(\GG\in \Omega)|m_N(z)  -\md(z)  |^{2p}]\lesssim \frac{1}{(N\eta)^{2p}}.
\end{align}
\item  If $\Re[z]\notin [-2-\eta, 2+\eta]$
\begin{align}\begin{split}\label{e:Qm_bound}
\bE\left[\bm1(\cG\in\Omega) \Im[m_N(z)-\md(z)]^{2p}\right]\lesssim \bE\left[\bm1(\GG\in \Omega)\left(\frac{1}{N^2\eta\kappa}+\frac{1}{(N\eta)^{3}\kappa^{1/2}}\right)^p\right].
\end{split}\end{align}
\end{enumerate}
\end{proposition}

\begin{proof}[Proof of \Cref{p:m-mdmoment}]
For simplicity of notation, we omit the dependence on $z$, and write $Y:=Y_\ell(Q(z),z)$ and $X:=X_\ell(Q(z),z)$. We recall from \Cref{c:rigidity} that for 
$|z-2|\leq N^{-\fg}$ and $\Im[z]\geq N^{-1+\fg}$, conditioned on $\cG\in \Omega$, with overwhelmingly high probability it holds
 \begin{align}\label{e:mNprior}
&|m_N -X|\lesssim  \frac{N^{2\fc}(\kappa+\eta)^{1/4}}{N\eta}+\frac{N^{6\fc}}{(N\eta)^{3/2}},\quad |Q-\msc|\lesssim  \frac{N^{8\fo}}{N\eta},
\end{align}
and  the moments of $Q-Y$ satisfy 
\begin{align}\label{e:goodQ-YQ}
&\bE[\bm1(\GG\in \Omega)|Q  -Y  |^{2p}]
\lesssim \bE\left[\bm1(\GG\in \Omega)\left(\Phi^{2p}+\left(\frac{\widetilde \Upsilon \Phi}{N\eta}\right)^p+\left(\frac{\Upsilon \Phi}{(d-1)^{\ell/4}N\eta}\right)^{p}\right)\right].
\end{align}
We divide the proof into the following three steps.

\textit{Step 1.} We start with some preparation work. By taking imaginary parts of \eqref{e:msc_equation} and \eqref{e:md_equation}, we get
\begin{align}\begin{split}\label{e:mscmd_2}
&\Im[\msc]=(1+\OO(\sqrt{|z-2|}))\Im[\sqrt{z-2}],\quad \Im[\md]=(1+\OO(\sqrt{|z-2|}))\cA\Im[\sqrt{z-2}],\\
&|1-\msc^2|\asymp \sqrt{|z-2|}, \quad \Im[1-\msc^2]=(1+\OO(\sqrt{|z-2|}))2\Im[\sqrt{z-2}].
\end{split}\end{align}
We have from \eqref{e:recurbound} and \eqref{e:Xrecurbound} that
\begin{align}\begin{split}\label{e:Q-YQ4}
&\phantom{{}={}}Q -Y =(Q -\msc)-(Y -\msc)
=(1-\msc^{2\ell+2})(Q -\msc)\\
&-\msc^{2\ell+2}\md\left(\frac{1-\msc^{2\ell+2}}{d-1}+\frac{d-2}{d-1}\frac{1-\msc^{2\ell+2}}{1-\msc^2}\right)(Q -\msc)^2+\OO(\ell^2 |Q -\msc|^3),
\end{split}
\end{align}
and  
\begin{align}\begin{split}\label{e:mt-Xt1}
m_N-\md
&=(m_N-X)+\frac{d}{d-1}\md^2\msc^{2\ell}(Q-\msc)+\OO\left(\ell |Q -\msc|^2\right).
\end{split}\end{align}
By \eqref{e:mNprior} and \eqref{e:mt-Xt1}, we see that conditioned on $\cG\in \Omega$, the following holds with overwhelmingly high probability
\begin{align}\label{e:MNmQm}
m_N-m_d =\frac{d}{d-1}\md^2\msc^{2\ell}(Q-\msc)+\OO\left( \frac{N^{2\fc}(\kappa+\eta)^{1/4}}{N\eta}+\frac{N^{6\fc}}{(N\eta)^{3/2}}\right).
\end{align}
Note that \eqref{e:mscmd_2} implies $\Im[\msc], \Im[\md]\lesssim \sqrt{|z-2|}\asymp\sqrt{\kappa+\eta}$. Hence taking imaginary part on \eqref{e:MNmQm} yields
\begin{align}\label{e:Immz2}
\Im[m_N-\md]=\frac{d}{d-1}\md^2\msc^{2\ell}\Im[Q-\msc]+\OO\left(\frac{N^{2\fc}(\kappa+\eta)^{1/4}}{N\eta}+\frac{N^{6\fc}}{(N\eta)^{3/2}}\right).
\end{align}
We can rearrange \eqref{e:Q-YQ4} as the following quadratic equation:
\begin{align}\label{e:Qt_quadratic}
(Q -\msc)^2 + b (Q -\msc)=(Q -\msc)(Q -\msc+b) =\delta,
\end{align}
where 
\begin{align}\begin{split}\label{e:defbdelta}
&b=-\frac{(1-\msc^2)}{\msc^{2\ell}\md\left(\frac{d-2}{d-1}+\frac{1-\msc^2}{d-1}\right)}
=-(1-\msc^2)\frac{(d-1)}{\msc^{2\ell}\md(d-1-\msc^2)},\quad |b|\asymp \sqrt{\kappa+\eta},\\
&|\delta|\lesssim \frac{|Q -Y |}{\ell} +\OO(\ell|Q -\msc|^3).
\end{split}\end{align}
By plugging \eqref{e:mscmd_2} into the expression of $b$, we get that 
\begin{align}\label{e:Imb}
\Im[b]= (1+\OO(\sqrt{|z-2|}))2\Im[\sqrt{z-2}].
\end{align}
It follows from \eqref{e:Qt_quadratic}, 
\begin{align}\label{e:Qt1}
\min\{|Q-\msc|, |Q+b-\msc|\}\lesssim \sqrt{|\delta|}.
\end{align}
Moreover, since $|b|\asymp \sqrt{\kappa+\eta}$ from \eqref{e:defbdelta}, $\max\{|Q-\msc|, |Q+b-\msc|\}\gtrsim \sqrt{\kappa+\eta}$. Together with \eqref{e:Qt_quadratic}, we conclude that
\begin{align}\label{e:Qt2}
\min\{|Q-\msc|, |Q+b-\msc|\}\lesssim \frac{|\delta|}{\sqrt{\kappa+\eta}}.
\end{align}
Combining \eqref{e:Qt1} and \eqref{e:Qt2} we have
\begin{align}\label{e:QQQ}
\min\{|Q-\msc|, |Q+b-\msc|\}\lesssim \min\left\{\sqrt{|\delta|}+\frac{|\delta|}{\sqrt{\kappa+\eta}}\right\} \lesssim \frac{|\delta|}{\sqrt{|\delta| +\kappa+\eta}}.
\end{align}
Moreover, since $\Im[Q]\geq 0$, using \eqref{e:mscmd_2} and \eqref{e:Imb} we get
\begin{align}\label{e:Qbmsc}
|Q+b-\msc|\geq\Im[Q+b-\msc] \geq \Im[b-\msc]= (1+\OO(\sqrt{|z-2|}))\Im[\sqrt{z-2}]\asymp \Im[b].
\end{align}
Thus
\begin{align*}
|\Im[Q-\msc]|\leq |Q+b-\msc|+|\Im[b]|\lesssim  |Q+b-\msc|,
\end{align*}
and \eqref{e:QQQ} implies
\begin{align}\label{e:ImQ}
\Im[Q-\msc]\lesssim \min\{|Q-\msc|, |Q+b-\msc|\}\lesssim  \frac{|\delta|}{\sqrt{|\delta| +\kappa+\eta}}.
\end{align}
We recall from \eqref{e:Immz2},
\begin{align}\begin{split}\label{e:Immz}
&\Im[m_N]\lesssim \Im[m_d]+|\Im[Q-\msc]|+\frac{N^{2\fc}(\kappa+\eta)^{1/4}}{N\eta}+\frac{N^{6\fc}}{(N\eta)^{3/2}}.\\
&\Phi=\frac{\Im[m_N ]}{N\eta}+\frac{1}{N^{1-2\fc}}\leq \frac{\Im[m_d]+|\Im[Q-\msc]|}{N\eta}+\frac{N^{2\fc}(\kappa+\eta)^{1/4}}{(N\eta)^2}+\frac{N^{6\fc}}{(N\eta)^{5/2}}+\frac{1}{N^{1-2\fc}}.
\end{split}\end{align}
Moreover, using \eqref{e:recurbound}, $|1-\del_1 Y_\ell(\msc(z  ),z  )|=|1-\msc^{2\ell}|\lesssim \ell|1-\msc^2|$, and 
\begin{align}\begin{split}\label{e:Upbound}
\Upsilon&=
|1-\del_1 Y_\ell(\msc(z  ), z  )|+\OO(\ell |Q  -\msc|)+(d-1)^{8\ell}\Phi\lesssim \ell (\sqrt{\kappa+\eta}+|Q  -\msc|),\\
\wt \Upsilon&= |1-\del_1 Y_\ell(\msc(z  ),z  )| +\OO(\ell|Q  -\msc|)\lesssim \ell( \sqrt{\kappa+\eta}+|Q  -\msc|).
\end{split}\end{align}

\textit{Step 2.} In the following, we first prove \eqref{e:moment1}. If $\Re[z]\in [-2-\eta, 2+\eta]$, we have by \eqref{e:Imb}
\begin{align*}
\Im[\sqrt{z-2}]\asymp \sqrt{\kappa+\eta}\asymp |b|\asymp \Im[b].
\end{align*}
And together with \eqref{e:Qbmsc}, we get
\begin{align*}
|Q-\msc|\lesssim |Q-\msc+b|+|b|\lesssim |Q-\msc+b|.
\end{align*}
We conclude that from \eqref{e:QQQ}
\begin{align}\label{e:Q-msc}
|Q-\msc|\lesssim \min\{|Q-\msc|, |Q+b-\msc|\}\lesssim  \frac{|\delta|}{\sqrt{|\delta| +\kappa+\eta}}.
\end{align}
Moreover, for $N\eta \sqrt{\kappa+\eta}\geq \fC$, using \eqref{e:imm_behavior} we also have
\begin{align*}
\Im[m_d]\asymp \sqrt{\kappa+\eta}\geq \frac{N^{2\fc}(\kappa+\eta)^{1/4}}{N\eta}+\frac{N^{6\fc}}{(N\eta)^{3/2}}, \quad \frac{\sqrt{\kappa+\eta}}{N\eta}\geq \frac{1}{N^{1-2\fc}}.
\end{align*}
The ecpression \eqref{e:Immz} simplifies to
\begin{align*}\begin{split}
&\Im[m_N ]\leq \sqrt{\kappa+\eta}+|\Im[Q-\msc]|\leq\sqrt{\kappa+\eta}+|Q-\msc|,\quad \Phi\lesssim \frac{\sqrt{\kappa+\eta}+|Q - \msc |}{N\eta}, \\
\end{split}\end{align*}
and 
\begin{align}\label{e:error}
&\phantom{{}={}}\Phi^2+\frac{\widetilde \Upsilon \Phi}{  \ell N\eta}+\frac{\Upsilon \Phi}{ \ell^2 (d-1)^{\ell/4}N\eta}\lesssim \left(\frac{\sqrt{\kappa+\eta}+|Q - \msc |}{N\eta}\right)^2.
\end{align}
By taking the $2p$-th moment on both sides of \eqref{e:Q-msc}, and recall $\delta$ from \eqref{e:defbdelta}, we get
\begin{align*}\begin{split}
\bE[\bm1(\GG\in \Omega)|Q  -\msc  |^{2p}]
&\lesssim \bE\left[\bm1(\GG\in \Omega) \frac{|\delta |^{2p}}{(|\delta|+\kappa+\eta)^p}\right]\\
&\lesssim \bE\left[\bm1(\GG\in \Omega) \frac{|Q -Y |^{2p}}{\ell^{2p}(\kappa+\eta)^p}+{(\ell^2|Q -\msc|^3)^{p}}\right].
\end{split}\end{align*}
Recall that from \eqref{eq:infbound0}, $|Q-\msc|\lesssim N^{-\fb}$ with overwhelmingly high probability. By rearranging the above estimate, we conclude 
\begin{align}\begin{split}\label{e:inspec}
\bE[\bm1(\GG\in \Omega)|Q  -\msc  |^{2p}]
&\lesssim \bE\left[\bm1(\GG\in \Omega) \frac{|Q -Y |^{2p}}{\ell^{2p}(\kappa+\eta)^p}\right]
\\
&\lesssim \bE\left[\bm1(\GG\in \Omega)\left(\Phi^2+\frac{\widetilde \Upsilon \Phi}{  \ell N\eta}+\frac{\Upsilon \Phi}{ \ell^2 (d-1)^{\ell/4}N\eta}\right)^{2}\right]\\
&\lesssim 
\bE\left[\bm1(\GG\in \Omega)\left(\frac{1}{(N\eta)^{2p}}+\left(\frac{ |Q-\msc|}{N\eta\sqrt{\kappa+\eta}}\right)^{2p}\right)\right],
\end{split}\end{align}
where the second statement follows from \eqref{e:goodQ-YQ}, and the third statement follows from \eqref{e:error}.
Noticing that $N\eta \sqrt{\kappa+\eta} \geq \fC$, we conclude from \eqref{e:MNmQm} and \eqref{e:inspec} that 
\begin{align*}
\bE[\bm1(\GG\in \Omega)|Q  -\msc  |^{2p}]\lesssim \frac{1}{(N\eta)^{2p}},\quad 
\bE[\bm1(\GG\in \Omega)|m_N  -\md  |^{2p}]\lesssim \frac{1}{(N\eta)^{2p}}.
\end{align*}
This finishes the proof of \eqref{e:moment1}.

\textit{Step 3.} Finally we prove \eqref{e:Qm_bound}. Outside of the spectrum, namely $\Re[z]\not\in [-2-\eta, 2+\eta]$ we have $\kappa\geq \eta$, and from \eqref{e:imm_behavior} $\Im[\msc(z)]\asymp \Im[\md(z)]\asymp \eta/\sqrt{\kappa}$. Moreover,  
\begin{align*}\begin{split}
&\Im[Q  (z)]\leq \Im[\msc(z  )]+|Q (z)- \msc(z  )|\leq \frac{\eta}{\sqrt{\kappa}}+|Q (z)- \msc(z  )|,\\
& \Im[Q(z)]\leq \Im[\msc(z)-b]+|Q (z)+b- \msc(z  )|\leq  \frac{\eta}{\sqrt{\kappa}}+ |Q (z)+b- \msc(z  )|,
\end{split}\end{align*}
and it follows 
\begin{align*}
 \Im[Q(z)]\lesssim \frac{\eta}{\sqrt{\kappa}}+\min\{|Q-\msc|, |Q+b-\msc|\}.
\end{align*}
In this case, we can simplify \eqref{e:Immz} and \eqref{e:Upbound} as
\begin{align}\begin{split}\label{e:Immz22}
&\Im[m_N]\lesssim \frac{\eta}{\sqrt{\kappa}}+|\Im[Q-\msc]|+\frac{N^{2\fc}\kappa^{1/4}}{N\eta}+\frac{N^{6\fc}}{(N\eta)^{3/2}} \lesssim |\Im[Q-\msc]|+ \frac{\eta}{\sqrt{\kappa}}+\frac{1}{N\eta},\\
&\Phi(z)\lesssim \frac{|\Im[Q-\msc]|}{N\eta}+\frac{1}{N\sqrt{\kappa}}+\frac{1}{(N\eta)^2},\\
&\Upsilon(z), 
\wt \Upsilon(z)\lesssim \ell( \sqrt{\kappa}+\min\{|Q  -\msc|, |Q+b-\msc|\}).
\end{split}\end{align}
Then it follows
\begin{align*}
&\phantom{{}={}}\Phi^2+\frac{\widetilde \Upsilon \Phi}{  \ell N\eta}+\frac{\Upsilon \Phi}{ \ell^2 (d-1)^{\ell/4}N\eta}
\lesssim \frac{\widetilde \Upsilon \Phi}{\ell N\eta}
\lesssim \frac{\min\{|Q  -\msc|, |Q+b-\msc|\}^2}{(N\eta)^2}\\
&+\frac{\sqrt{\kappa}}{N\eta}\frac{\min\{|Q  -\msc|, |Q+b-\msc|\}}{N\eta}
+\frac{\sqrt{\kappa}}{N\eta}\left(\frac{1}{N\sqrt{\kappa}}+\frac{1}{(N\eta)^{2}}
\right)\\
&\lesssim \frac{\min\{|Q  -\msc|, |Q+b-\msc|\}^2}{(N\eta)^2}
+\frac{\kappa^{1/2}\min\{|Q  -\msc|, |Q+b-\msc|\}^2}{N\eta}+\frac{\kappa^{1/2}}{(N\eta)^3}+\frac{1}{N^2\eta},
\end{align*}
and \eqref{e:goodQ-YQ} leads to
\begin{align}\begin{split}\label{e:goodQ-YQ2}
\bE\left[\bm1(\GG\in \Omega)\frac{|Q  -Y  |^{2p}}{\ell^{2p}(\kappa+\eta)^p}\right]
&\lesssim \bE\left[\bm1(\GG\in \Omega)\frac{\min\{|Q  -\msc|, |Q+b-\msc|\}^{2p}}{(N\eta \sqrt{\kappa})^p}\right]\\
&+\bE\left[\bm1(\GG\in \Omega)\left(\frac{1}{N^2\eta\kappa}+\frac{1}{(N\eta)^{3}\kappa^{1/2}}
\right)^p\right].
\end{split}\end{align}
We also have the following bound thanks to \eqref{e:mNprior}
\begin{align}\label{e:goodQ-YQ3}
\bE\left[\bm1(\GG\in \Omega) \frac{(\ell|Q-\msc|^3)^{2p}}{(\kappa)^p}\right]
\lesssim \bE\left[\bm1(\GG\in \Omega) \left(\frac{N^{8\fo}}{(N\eta)}\right)^{6p}\frac{\ell^{2p}}{\kappa^p}\right]\lesssim \bE\left[\bm1(\GG\in \Omega) \frac{1}{(N\eta)^{4p}\kappa^p}\right].
\end{align}
By taking the $2p$-th moment on both sides of \eqref{e:ImQ}, we get
\begin{align*}
&\phantom{{}={}}\bE\left[\bm1(\cG\in\Omega) \min\{|Q-\msc|, |Q+b-\msc|\}^{2p}\right]\lesssim \bE\left[\bm1(\GG\in \Omega) \frac{(|Q -Y |/\ell)^{2p}+(\ell|Q-\msc|^3)^{2p}}{(\kappa)^p}\right]\\
&\lesssim \bE\left[\bm1(\GG\in \Omega)\frac{\min\{|Q  -\msc|, |Q+b-\msc|\}^{2p}}{(N\eta \sqrt{\kappa})^p}\right]+\bE\left[\bm1(\GG\in \Omega)\left(\frac{1}{N^2\eta\kappa}+\frac{1}{(N\eta)^{3}\kappa^{1/2}}\right)^p\right],
\end{align*}
where in the last line we used \eqref{e:goodQ-YQ2}, \eqref{e:goodQ-YQ3}, and $N\eta\sqrt{\kappa}\geq \fC$.
By rearranging, we conclude 
\begin{align*}
\bE\left[\bm1(\cG\in\Omega) \Im[Q-\msc]^{2p}\right]&\lesssim\bE\left[\bm1(\cG\in\Omega) \min\{|Q-\msc|, |Q+b-\msc|\}\right]\\
&\lesssim\bE\left[\bm1(\GG\in \Omega)\left(\frac{1}{N^2\eta\kappa}+\frac{1}{(N\eta)^{3}\kappa^{1/2}}\right)^p\right],
\end{align*}
and 
\begin{align*}
\bE\left[\bm1(\cG\in\Omega) \Im[m_N-\md]^{2p}\right]\lesssim \bE\left[\bm1(\GG\in \Omega)\left(\frac{1}{N^2\eta\kappa}+\frac{1}{(N\eta)^{3}\kappa^{1/2}}\right)^p\right].
\end{align*}
This finishes the proof of \eqref{e:Qm_bound}.
\end{proof}

\subsection{Proofs of \Cref{p:sti_tightness} and \Cref{p:eig_concentration}}

\begin{proof}[Proof of \Cref{p:sti_tightness}]
From \eqref{e:moment1} and \eqref{e:Qm_bound}, we have for any $N\eta\sqrt{\kappa+\eta}\geq \fC$, conditioned on $\cG\in \Omega$, with probability at least $1-\OO(M^{-{D+1}})$, the following holds
\begin{align}\label{e:mdif1}
|m_N(z)-\md(z)|\leq \frac{M}{N\Im[z]}.
\end{align}

Next we show by an union bound that conditioned on $\cG\in \Omega$, with probability at least $1-\OO(M^{-D})$, \eqref{e:mdiff_tight} holds. We take a lattice 
\begin{align}\label{e:defmbfL}
{\mathbf L}=\{z=\gamma/4(\bZ+\ri\bZ)\in \bC^+: \gamma\leq \Im[z]\leq 2(M/N)^{2/3},  |\Re[z]-2|\leq 2\gamma\}.
\end{align}
By an union bound, we have that conditioned on $\cG\in \Omega$, \eqref{e:mdif1} holds for all $z\in {\mathbf L}$, with probability at least $1-\OO(M^{-D})$.

We notice the following Lipschitz property of $m_N$
\begin{align}\label{e:dz}
|\del_z m_N(z)|\leq \frac{\Im[m_N(z)]}{\Im[z]}.
\end{align}
For any $z\in {\mathbf L}$, and any $|z'-z|\leq \gamma/2$, we have $\Im[z], \Im[z']\geq \gamma/2$. By integrating \eqref{e:dz} from $z$ to $z'$, 
\begin{align}\label{e:Imbound}
\Im[m_N(z')]\lesssim e^{2|z'-z|/\gamma}\Im[m_N(z)]\lesssim \Im[m_N(z)].
\end{align}
And similarly we also have $\Im[\md(z')]\lesssim \Im[\md(z)]$. Then
we have
\begin{align}\begin{split}\label{e:mmdd} 
&\phantom{{}={}}|(m_N(z')-\md(z'))-(m_N(z)-\md(z))|
\lesssim |m_N(z')-m_N(z)|+|\md(z')-\md(z))|\\
&\lesssim \int_{z}^{z'}\frac{(\Im[m_N(w)]+\Im[\md(w)])|\rd w|}{\Im[w]}\rd w\lesssim \frac{|z'-z|(\Im[m_N(z)]+\Im[\md(z)])}{\Im[z]}\\
&\lesssim\frac{(\gamma/2)(\Im[m_N(z)-\md(z)]+2\Im[\md(z)])}{\Im[z]}  \lesssim \frac{M}{N\eta}+\frac{\gamma\Im[\md(z)]}{\Im[z]}
\lesssim \frac{M}{N\eta}+\frac{\gamma\sqrt{|z-2|}}{\Im[z]}
\lesssim \frac{M}{N\eta},
\end{split}\end{align}
where the second line follows from integrating \eqref{e:dz} and \eqref{e:Imbound}; the last line follows from \eqref{e:mdif1}, $\Im[m_d(z)]\lesssim \sqrt{|z-2|}$ from \eqref{e:imm_behavior} and our assumption $\gamma\sqrt{|z-2|}\lesssim \gamma (M/N)^{1/3}\lesssim M/N$. 

By our construction of ${\mathbf L}$ in \eqref{e:defmbfL}, for a $z'$ such that $ |\Re[z']-2|\leq 2\gamma, \; \gamma\leq \Im[z']\leq 2(M/N)^{2/3}$ there exists a $z\in {\mathbf L}$ such that $|z'-z|\leq \gamma/2$. Thus we conclude from \eqref{e:mdif1} and \eqref{e:mmdd} that \eqref{e:mdiff_tight} holds with probability at least $1-\OO(M^{-D})$.

Next we prove \eqref{e:mdiff_tight2}. As a consequence of \eqref{e:mdiff_tight}, for these $z$ such that $ |\Re[z]-2|\leq 2\gamma, \; \gamma\leq \Im[z]\leq 2(M/N)^{2/3}$ we also have that
\begin{align}\label{e:mdif2}
|\Im[m_N(z)-\md(z)]|\leq \frac{M}{N\eta}.
\end{align}
Notice that $y \Im[m_N(x+\ri y)]$ is monotone increasing in $y$.
If  for some $z'=E+\ri \eta'$ such that $|\Im[m_N(z')-\md(z')]|, \Im[\md(z')]\leq M/N\eta'$, then for any $\eta\leq \eta'$ and $z=E+\ri\eta$,
\begin{align}\begin{split}\label{e:below_the_curve}
|\Im[m_N(z)-\md(z)]|
&\leq \Im[m_N(z)]+\Im[\md(z)]
\leq \frac{\eta'}{\eta}\Im[m_N(z')]+\Im[\md(z)]\\
&\leq \frac{\eta'}{\eta}|\Im[m_N(z')]-\Im[\md(z')]|
+\frac{2\eta'}{\eta}\Im[\md(z')]
\lesssim \frac{M}{N\eta}.
\end{split}\end{align}
Thus if \eqref{e:mdif2} holds for $z$ such that $ |\Re[z]-2|\leq 2\gamma, \; \gamma\leq \Im[z]\leq 2(M/N)^{2/3}$, then for any $2-2\gamma\leq x\leq 2+2\gamma$, we can take $z=x+\ri \gamma$. Then it holds $\Im[m_N(z)-\md(z)]\leq M/N\Im[z]$, and from \eqref{e:imm_behavior} $\Im[\md(z)]\lesssim \sqrt{|z-2|}\lesssim \sqrt{\gamma}\leq M/(N\gamma)=M/(N\Im[z])$. Then \eqref{e:below_the_curve} implies that \eqref{e:mdiff_tight2} holds. This finishes the proof.
\end{proof}

To prove \Cref{lem:parti-clo}, we require the following lemma, which translates estimates on the Stieltjes transform into an estimate for the linear eigenvalue statistics. This result is a slight modification of \cite[Lemma 3.7]{bourgade2022optimal}, we sketch its proof in \Cref{app:Green}.
\begin{lemma}  \label{lem:parti-clo}
There exists a large constant $\fC>0$ such that the following holds. Assume that for some $\gamma>0$ and $M>0$, with  $2-\gamma\leq x\leq 2+\gamma$, the following bounds hold
\begin{align}\begin{split}\label{e:mNbound}
&|m_N(x+\ri y)-\md(x+\ri y)|\leq M/N y,\text{ for } \eta<y\leq
 2\gamma,\\
&|\Im[m_N(x+\ri y)-\md(x+\ri y)]|\leq M/N y,\text{ for } 0<y\leq 2\gamma.
\end{split}\end{align}
Then for any twice differentiable function $f$ supported on $[2-\gamma, 2+\gamma]$, we have
\begin{align}\label{e:intf}
\left|\sum_{i=1}^N f(\lambda_i)-N\int_\bR f(x) \varrho_d(x)\rd x\right| \leq \fC M\left(\frac{\|f\|_1}{\gamma}+\eta \|f''\|_1 +\log(\gamma/\eta)\|f'\|_1\right).
\end{align}
\end{lemma}

\begin{proof}[Proof of \Cref{p:eig_concentration}]
We consider an interval $I=[2+\kappa, 2+\kappa+\eta]$ for some $0<\eta\leq \kappa$. Taking $z=2+\kappa+\ri\eta$, then 
\begin{align}\label{e:number_particle}
\hspace{-0.16cm}\#\{s: \lambda_s\in I\}\leq \sum_{s: \lambda_s\in I }\frac{2\eta^2}{|z-\lambda_s|^2}\leq  2\eta N\Im[m_N(z)]\leq 2\eta N|\Im[m_N(z)-\md(z)]|+2\eta N\Im[\md(z)].
\end{align}
By \eqref{e:imm_behavior}, we have $2\eta N\Im[\md(z)]\asymp 2\eta^2 N/\sqrt{\kappa+\eta}\lesssim2\eta^2 N/\sqrt{\kappa} \leq 1/2$, provided we take $\eta\leq \fC_0^{-1} \kappa^{1/4}/\sqrt{N}$, for some large $\fC_0>0$. Notice that $\#\{s: \lambda_s\in I\}$ is an integer, we can rewrite the estimate \eqref{e:number_particle} as
\begin{align*}
\#\{s: \lambda_s\in I\}\leq 4\eta N|\Im[m_N(z)-\md(z)]|.
\end{align*}

Together with \eqref{e:Qm_bound}, we conclude that for $\fC/(N\sqrt{\kappa})\leq \eta\leq \fC_0^{-1} \kappa^{1/4}/\sqrt{N}$ and $0\leq \eta\leq \kappa$,
\begin{align*}\begin{split}
\bP( \#\{s: \lambda_s\in I\}\geq 1, \cG\in \Omega)\leq \bE\left[\bm1(\cG\in\Omega) \#\{s:\lambda_s\in I\}^{2p}\right]
&\lesssim 
\bE\left[\bm1(\GG\in \Omega)\left(\frac{\eta^p}{\kappa^p}+\frac{1}{(N\eta)^{p}\kappa^{p/2}}\right)\right].
\end{split}\end{align*}
Thus for any $\fC^4_0\fC N^{-2/3}\leq \kappa\leq N^{-1/2}$ and $\eta=\fC_0^{-1} \kappa^{1/4}/\sqrt{N}\geq \fC^{1/3}N^{-2/3}$ it holds
\begin{align*}
\bP( \#\{s: \lambda_s\in [2+\kappa, 2+\kappa+\eta ]\}\geq 1, \cG\in \Omega)\lesssim \frac{1}{(N^{2/3}\kappa)^{3p/4}}.
\end{align*}
And it follows that 
\begin{align}\label{e:closela}
\bP( \#\{s:\lambda_s\in [2+\kappa, 2+\kappa+N^{-2/3}]\}\geq 1, \cG\in \Omega)\lesssim \frac{1}{(N^{2/3}\kappa)^{3p/4}}.
\end{align}
Then \eqref{e:lambda2bound} follows from an union bound: for $t\geq \fC_0^4 \fC N^{-2/3}$, 
\begin{align*}
&\phantom{{}={}}\bP(N^{2/3}(\lambda_2-2)\geq t , \cG\in \Omega)\\
&\leq \sum_{t\leq j< 2N^{1/6}}\bP(N^{2/3}(\lambda_2-2)\in[j, j+1], \cG\in \Omega)
+\bP(N^{2/3}(\lambda_2-2)\geq N^{1/6}, \cG\in \Omega)\\
&\lesssim \sum_{t\leq j< 2N^{1/6}} \frac{1}{j^{3p/4}}+\frac{1}{N^D}\lesssim \frac{1}{t^D},
\end{align*}
where the last line follows from \eqref{e:closela} and the optimal eigenvalue rigidity \eqref{e:optimal_rigidity}.

The statement \eqref{e:lambdajbound} is vacuous for small $x$. For $x\geq N^{1/3}$, it follows directly from \eqref{e:optimal_rigidity}.
For the rest cases of \eqref{e:lambdajbound}, we can use  \Cref{lem:parti-clo} for $\gamma=x N^{-2/3}$. Take $f(x)=1$ for $x\in[2-\gamma,  2+\gamma]$, and $f=0$ for $x\not\in [2-2\gamma,  2+2\gamma]$, such that $|f'|\lesssim 1/\gamma, |f''|\lesssim 1/\gamma^2$. Let $M= x^{3/2}$, then $N^{1/3}\geq (M/N)^{2/3}\geq \gamma=xN^{-2/3}\geq \fC N^{-2/3}$, and
 \Cref{p:sti_tightness} verifies the assumptions in \Cref{lem:parti-clo}. We conclude that with probability at least $1-\OO(1/M^D)=1-\OO(1/x^D)$, the following holds
\begin{align*}
\#\{i: \lambda_i\in [2-\gamma, 2+\gamma]\}\leq \sum_{i=1}^N f(\lambda_i)=N\int_\bR f(x) \varrho_d(x)\rd x +\OO ( M)=\OO(N\gamma^{3/2}+M)=\OO(N\gamma^{3/2}),
\end{align*}
and \eqref{e:lambdajbound} holds by rearranging.

\end{proof}

\appendix
\section{Properties of the Green's functions}

Throughout this paper, we repeatedly use some (well-known) identities for Green's functions,
which we collect in this appendix.

%

\subsection{Schur complement formula}

Given an $N\times N$ matrix $H$ and an index set $\bT \subset \qq{N}$, recall that we denote by
$H|_\bT$ the $\bT \times \bT$-matrix obtained by restricting $H$ to $\bT$,
and that by $H^{(\bT)} = H|_{\bT^\complement}$ the matrix obtained by removing
the rows and columns corresponding to indices in $\bT$.
Thus, for any $\bT \subset \qq{N}$,
any symmetric matrix $H$ can be written (up to rearrangement of indices) in the block form
\begin{equation*}
  H = \begin{bmatrix} A& B^\top\\ B &D  \end{bmatrix},
\end{equation*}
with $A=H|_{\bT}$ and $D=H^{(\bT)}$.
The Schur complement formula asserts that, for any $z\in \bC^+$,
\begin{equation} \label{e:Schur}
 G=(H-z)^{-1}= \begin{bmatrix}
   (A-B^\top G^{(\bT)} B)^{-1} & -(A-B^\top G^{(\bT)} B)^{-1}B^\top G^{(\bT)} \\
   -G^{(\bT)} B(A-B^\top G^{(\bT)} B)^{-1} & G^{(\bT)}+G^{(\bT)} B(A-B^\top G^{(\bT)} B)^{-1}B^\top G^{(\bT)} 
 \end{bmatrix},
\end{equation}
where $G^{(\bT)}=(D-z)^{-1}$.
Throughout the paper, we often use the following special cases of \eqref{e:Schur}:
\begin{align} \begin{split}\label{e:Schur1}
  G|_{\bT} &= (A-B^\top G^{(\bT)} B)^{-1},\\
  G|_{\bT\bT^\complement}&=-G|_{\bT}B^\top G^{(\bT)},\\
   G|_{\bT^\complement}&=G^{(\bT)}+G|_{\bT^\complement\bT}(G|_{\bT})^{-1}G|_{\bT\bT^\complement}=G^{(\bT)}-G^{(\bT)}BG|_{\bT\bT^\complement},
  \end{split}
\end{align}
as well as the special case
\begin{equation} \label{e:Schurixj}
G_{ij}^{(k)} = G_{ij}-\frac{G_{ik}G_{kj}}{G_{kk}}
=G_{ij}+(G^{(k)}H)_{ik} G_{kj}.
\end{equation}


\subsection{Ward identity}

For any symmetric $N\times N$ matrix $H$, its Green's function $G(z)=(H-z)^{-1}$ satisfies
the Ward identity
\begin{equation} \label{e:Ward}
  \sum_{j=1}^{N} |G_{ij}(z)|^2= \frac{\Im G_{jj}(z)}{\eta},
\end{equation}
where $\eta=\Im [z]$. This identity provides a bound for the sum $\sum_{j=1}^{N} |G_{ij}(z)|^2$
in terms of the diagonal entries of the Green's function.

\section{Proof of \Cref{l:boundary_measure},  Proposition \ref{p:recurbound} and \Cref{lem:parti-clo}}
\label{app:Green}
\begin{proof}[Proof of \Cref{l:boundary_measure}]

For any compactly supported smooth function $f$ with $\supp(f)\in [-M, M]$, we show that for any small $y> 0$,
\begin{align}\begin{split}\label{e:fdiff}
&\left|\int _\bR f(x)\rd\mu_n(x)- \int _\bR f(x)\rd\mu(x)\right|\leq \int|u_n(x+\ri y)-u(x+\ri y)||f(x)|\rd x\\
&+10M y\|f\|_\infty\int_\bR \frac{\rd |\mu|(s)+\rd |\mu_n|(s)}{1+s^2}
+(y\ln(M/y)+y\|f\|_\infty/M)\int_{-2M}^{2M}(\rd|\mu|(s)+\rd|\mu_n|(s)),
\end{split}\end{align} 
and then the claim follows by sending $n\rightarrow \infty$ and $y\rightarrow 0+$. 

The statement \eqref{e:fdiff} follows from showing the following estimate and its analogue by replacing $u(x+\ri y)$ and $\mu$ with $u_n(x+\ri y)$ and $\mu_n$
\begin{align}\begin{split}\label{e:shift_diff}
&\phantom{{}={}}\left|\int_\bR u(x+\ri y) f(x)\rd x-\int_\bR  f(x)\rd\mu(x)\right|\\
&\leq 10M y\|f\|_\infty\int_\bR \frac{\rd |\mu|(s)}{1+s^2}
+\left(y\ln\left(\frac{M}{y}\right)+\frac{y\|f\|_\infty}{M}\right)\int_{-2M}^{2M}\rd|\mu|(s).
\end{split}\end{align}

We can rewrite the LHS of \eqref{e:shift_diff} as 
\begin{align}\label{e:diyi}
\int_\bR u(x+\ri y) f(x)\rd x-\int_\bR  f(x)\rd\mu(x)
&=\int\left(\frac{1}{\pi}\int_\bR \frac{y f(x) \rd x}{(x-s)^2+y^2}-f(s)\right)  \rd \mu(s).\end{align}
For $|s|\geq 2M$, we have 
\begin{align}\begin{split}\label{e:dier}
\frac{1}{\pi}\int_\bR \frac{y f(x) \rd x}{(x-s)^2+y^2}
\leq \frac{1}{\pi}\int_{-M}^M \frac{y \|f\|_\infty \rd x}{(x-s)^2+y^2}
\leq \frac{10 My\|f\|_\infty }{1+s^2}.
\end{split}\end{align}
For $|s|\leq 2M$
\begin{align}\begin{split}\label{e:disan}
&\phantom{{}={}}\frac{1}{\pi}\int_{-4M}^{4M} \frac{y f(x) \rd x}{(x-s)^2+y^2}-f(s)
=\frac{1}{\pi}\int_{-4M}^{4M} \frac{y (f(s)+\OO(|x-s|)) \rd x}{(x-s)^2+y^2}-f(s)\\
&\lesssim 
y\int_{-4M}^{4M}\frac{|x-s|\rd x}{(x-s)^2+y^2}+f(s)\int_{x\in \bR\setminus[-4M, 4M]}\frac{y}{(x-s)^2+y^2}\rd y\lesssim
y\ln(M/y)+\frac{yf(s)}{M}.
\end{split}\end{align}
The claim \eqref{e:shift_diff} follows from plugging \eqref{e:dier} and \eqref{e:disan} into \eqref{e:diyi}.
\end{proof}

\begin{proof}[Proof of Proposition \ref{p:recurbound}]
The proofs for $X_\ell$ and $Y_\ell$ are identical, so we will only provide the proof for $Y_\ell$. The first claim \eqref{e:Yl_derivative} follows from the Taylor expansion expression \eqref{e:recurbound}. 

To prove \eqref{e:recurbound}, we denote $\cH=\cB_\ell(o, \cY_d)$, which is the truncated $(d-1)$-ary tree at level $\ell$. We denote its vertex set as $\bH$ and normalized adjacency matrix as $H$. We denote $\bI$ and $\bI^\del$ the diagonal matrices, such that for $x,y\in \bH$,  $\bI_{xy}=\delta_{xy}$ and $\bI^\del_{xy}=\bm1(\dist_\cH(x,o)=\ell)\delta_{xy}$. Then \eqref{e:defP} gives that
\begin{align*}
Y_\ell( \Delta, z)=P(\cH,z,\Delta)
=\left(H-z-\Delta \mathbb I^\del\right)^{-1},\quad P=P(\cH,z,\msc(z))=(H-z-\msc(z)\bI^\del)^{-1}.
\end{align*}
In the rest of the proof, we will simply write $\msc=\msc(z), \md=\md(z)$. We can compute the Green's function $P(\cH, w, \Delta)$ by a perturbation argument,
\begin{align}\begin{split}\label{e:expansionGP}
P(\cH, z, \Delta)
&=\left(H-z-\Delta \mathbb I^\del\right)^{-1}=\left(H-z-\msc\mathbb I^\del-(\Delta-\msc)\mathbb I^\del\right)^{-1}\\
&=P+P\sum_{k\geq 1}((\Delta-\msc)\mathbb I^\del P)^k.
\end{split}\end{align}
With the explicit expression of $P$ as given in \eqref{e:Gtreemsc2}, we can compute
\begin{align}\label{e:PBPoo}
\left(P(\Delta-\msc)\mathbb I^\del P\right)_{oo}
&=\msc^{2\ell+2}(\Delta-\msc).
\end{align}
Moreover, for $k\geq 2$ we will show the following two relations for $P$,
\begin{align}
   \label{e:Pboundary} &\left(P\mathbb I^\del P\mathbb I^\del P\right)_{oo}=\msc^{2\ell+2}\md\left(\frac{1-\msc^{2\ell+2}}{d-1}+\frac{d-2}{d-1}\frac{1-\msc^{2\ell+2}}{1-\msc^2}\right) ,\\
    \label{e:Ptotalsum}&\left(P(\mathbb I^\del P)^{k}\right)_{oo}\lesssim (C\ell)^{k-1}.
\end{align}
The claim \eqref{e:recurbound} follows from plugging \eqref{e:Pboundary} and \eqref{e:Ptotalsum} into \eqref{e:expansionGP}.

We recall that for $l,l'\in \bH$, ${\rm anc}(l,l')$ is the distance from the common ancestor of the vertices $l,l'$ to the root $o$ in $\cH$.
The relation \eqref{e:Pboundary} follows from explicit computation using \eqref{e:Gtreemsc} and \eqref{e:Gtreemsc2}
\begin{align*}
&\phantom{{}={}}\sum_{l, l'}P_{ol}P_{ll'}P_{l'o}=\sum_{l,l'}\msc^2 \left(-\frac{\msc}{\sqrt{d-1}}\right)^{2\ell}\md\left(1-\left(-\frac{\msc}{\sqrt{d-1}}\right)^{2+2{\rm anc}(l, l')}\right)\left(-\frac{\msc}{\sqrt{d-1}}\right)^{\dist_{\cH}(l,l')}\\
&= \frac{\msc^{2\ell+2}}{(d-1)^{2\ell}}\sum_{l,l'} \md\left(1-\left(\frac{\msc}{\sqrt{d-1}}\right)^{2+2{\rm anc}(l, l')}\right)\left(-\frac{\msc}{\sqrt{d-1}}\right)^{\dist_{\cH}(l,l')}\\
&=\msc^{2\ell+2}\md\left(1-\left(\frac{\msc}{\sqrt{d-1}}\right)^{2+2\ell}
+\sum_{r=1}^\ell \left(1-\left(\frac{\msc}{\sqrt{d-1}}\right)^{2+2(\ell-r)}\right)\left(\frac{\msc}{\sqrt{d-1}}\right)^{2r}(d-2)(d-1)^{r-1}
\right)\\
&=\msc^{2\ell+2}\md\left(\frac{1-\msc^{2\ell+2}}{d-1}+\frac{d-2}{d-1}\frac{1-\msc^{2\ell+2}}{1-\msc^2}\right),
\end{align*}
where the summation in the first line is over $l,l'$ such that $\dist_{\cH}(l,o)=\dist_{\cH}(l',o)=\ell$; in the second to last line we used that for a given $l$, there are $(d-2)(d-1)^{r-1}$ values of $l'$ such that $\dist_{\cH}(l, l')=2r, {\rm anc}(l, l')=\ell-r$ for $1\leq r\leq \ell$. When $l=l'$, it holds $\dist_{\cH}(l, l')=0, {\rm anc}(l, l')=\ell$.

For each $i,j\in \bH$, let $|P|_{ij}\deq |P_{ij}|$.
From the above computation, for any $l'$ such that $\dist_{\cH}(l',o)=\ell$, we have
\begin{align*}
\sum_l |P_{ll'}|\lesssim \ell.
\end{align*}
where the summation in the first term is over $l$ such that $\dist_{\cH}(l,o)=\ell$. Then it follows 
\begin{align*}
\left(P(\mathbb I^\del P)^{k}\right)_{oo}
\lesssim \sum_{l_1, l_2,\cdots, l_k} |P_{ol_1}| |P_{l_1 l_2}|\cdots |P_{l_k o}|
\lesssim (C\ell)^{k-1}.
\end{align*}
\end{proof}

\begin{proof}[Proof of \Cref{lem:parti-clo}]
Take another smooth test function $\chi=\chi^{(s)}:\bR\to\bR_{\ge 0}$, with $\chi=1$ on $[-\gamma, \gamma]$, $\chi=0$ on $\bR\setminus [-2\gamma, 2\gamma]$, and $|\chi'|\lesssim 100/\gamma$ on $\bR$. 
Let $\tilde{f}(x+\ri y)=(f(x)+\ri y f'(x))\chi(y)$ for any $x,y\in\bR$.
Then by the Helffer–Sj{\" o}strand formula \cite[Section 11.2]{erdHos2017dynamical}, we have
\begin{align}\begin{split}\label{e:HS}
&\left|\sum_{i=1}^N f(\lambda_i)-N\int_\bR f(x) \varrho_d(x)\rd x\right|
\leq \fC(\mbox{I+II+III+IV})\\
&\mbox{I}=N\iint_{y\geq 0} (|f(x)+ y |f'(x)|)|\chi'(y)| |m_N(x+\ri y)-\md(x+\ri y)|\rd x\rd y\\
&\mbox{II}= N\iint_{0\leq y\leq \eta} f''(x)y\chi(y) |\Im[m_N(x+\ri y)-\md(x+\ri y)]|\rd x\rd y\\
&\mbox{III}= N\iint_{\eta\leq y} |f'(x)||\del_y(y\chi(y))| |m_N(x+\ri y)-\md(x+\ri y)|\rd x\rd y\\
&\mbox{IV}= N\int_\bR \eta |f'(x)|  |m_N(x+\ri\eta)-\md(x+\ri\eta)|\rd x.
\end{split}\end{align}
The claim \eqref{e:intf} follows from plugging \eqref{e:mNbound} into \eqref{e:HS}.
\end{proof}

\bibliography{ref}{}
\bibliographystyle{abbrv}

\end{document}